\documentclass[12pt,twoside,reqno]{amsart}
\usepackage{amsmath, amsfonts, amsthm, amssymb,amscd, graphicx, amscd}
\usepackage{float,epsf}
\usepackage[english]{babel}
\usepackage{enumerate}
\usepackage{tikz}

\usepackage{srcltx}
\usepackage{geometry}
\usepackage{verbatim}
\usepackage{mathrsfs}
\usepackage{hyperref}
\usepackage{enumitem} 
\usepackage{esint}

\newcommand{\R}{\mathbb{R}}

\newcommand{\tn}{\textnormal}

\newcommand{\eps}{\varepsilon}

\newtheorem{theorem}{Theorem}[section]
\newtheorem{lemma}[theorem]{Lemma}

\newtheorem{definition}[theorem]{Definition}

\newtheorem{proposition}[theorem]{Proposition}
\newtheorem{corollary}[theorem]{Corollary}
\newtheorem{remark}[theorem]{Remark}
\newtheorem{example}[theorem]{Example}
\newtheorem{notation}[theorem]{Notation}

\numberwithin{equation}{section}
\numberwithin{figure}{section}

\newcommand{\norm}[1]{\left\|#1\right\|}

\newcommand{\FF}{{\boldsymbol F}}

\renewcommand{\div}{\mathrm{div}} 

\renewcommand{\div}{\text{\sl div}\,} 

\def\intave#1{\int_{#1}\hbox{\llap{$\raise2.3pt\hbox{\vrule
height.9pt width7pt}\phantom{\scriptstyle{#1}}\mkern-2mu$}}}

\begin{document}
\title[On Two Species Long Range Segregation in an Annular Domain]{On Two Species Long Range Segregation in an Annular Domain}

\author{Howen Chuah}
\address{Howen Chuah,
Department of Mathematics, Purdue University,
 150 N. University Street,
 West Lafayette, IN 47907-2067, USA}
\email{hchuah@purdue.edu}

\keywords{Laplace operators, segregation models, free boundary problems, free boundary conditions, regular points}
\subjclass[2020]{Primary: 35J47, 35R35;
Secondary: 35Q92
}

\maketitle

\begin{abstract}

We consider a system of elliptic equations, depending on a small parameter $\epsilon > 0$, which models the long range segregation of populations. The system has been previously studied in \cite{CL2} for the regularity of the free boundary in dimension $2$ and in \cite{ChPaTo26_2} and \cite{ChPaTo26_3} for the partial regularity of the free boundary in higher dimensions. In this paper, we consider the special case with $K = 2$ populations in an annular domain in arbitrary dimensions. Using the uniqueness and the rotational invariance of the solution, we show that the free boundary consist of concentric spheres. Moreover, by an application of the free boundary condition derived in \cite{CL2}, we show that the free boundary is uniquely determined. We also examine how the radius of the free boundary change according to the domain, the interaction distance and the boundary data. In particular, we show that the free boundary converges to that of the adjacent model as the interaction distance tends to zero.

We also study the one parameter family of the elliptic system, in which the domain, the interaction distance, and the boundary data depends on $t$. We derive an ODE for the radius of the free boundary in dimension $2$ and dimension $n \geq 3$ separately. Several examples are given and discussed.

\,\\\end{abstract}

\vspace{-6mm}
\begingroup\hypersetup{hidelinks} 
\endgroup
\newcommand{\sizeofboundary}{R}

\section{Introduction}

In this paper, we study the following system of elliptic equations that models the long-range segregation of populations:  

\begin{equation}
\label{main_problems}
\begin{cases} \Delta u^{\varepsilon}_{i}= \frac{1}{\eps^2} u^{\eps}_i  \sum_{j \neq i}  H_R(u^{\eps}_j)(x)\quad & \text{ in } \Omega,\\
        u^{\eps}_i= f_i & \text{ on } (\partial \Omega)_{\leq {\sizeofboundary}}, \\
        
        u^{\eps}_i \geq 0 & \text{ in } \Omega \cup (\partial \Omega)_{\leq \sizeofboundary},
       \end{cases}
\end{equation}
for  $i=1,\ldots, K$, 
where $\Omega$ is a bounded Lipschitz domain in  $\R^n$, $\eps >0$, $0 < \sizeofboundary \leq 1$. The boundary neighborhood is defined as 
\begin{equation*}
    (\partial \Omega)_{\leq \sizeofboundary}:= \{x \in \Omega^c: d(x, \partial \Omega)\leq \sizeofboundary\} ,
\end{equation*}
where $d(x,\partial \Omega) := \inf_{y \in \partial \Omega}|x-y|$ denotes the Euclidean distance from $x$ to $\partial \Omega$, and $\Omega^c := \mathbb{R}^n \setminus \Omega$. 
The boundary data are nonnegative H{\"o}lder continuous functions
with supports separated by at least distance $\sizeofboundary$. Also, we assume 
\begin{equation}
\label{H1}
H_{\sizeofboundary}(w)(x) = \fint_{B_R(x)} w(y)dy
\end{equation}. 

We assume that the boundary data satisfy  
\begin{equation}    
\label{fiassumption}
        f_i : (\partial \Omega)_{\leq R} \rightarrow \mathbb{R}, \; f_i \geq 0, \; f_i \neq 0, \quad \textnormal{$f_i$ is H\"older continuous,} 
    \end{equation}
and that there is a constant $c > 0$ such that, for any $x \in \partial \Omega \cap \text{supp } f_i$,
\begin{equation}
\label{size_of_ball_intersect_support}
|{B}_r(x) \cap \text {supp }f_i| \geq c\,|{B}_r(x)|,
\end{equation}
and
       \begin{equation}\label{f_i-f_j_disjointsupport}       
        d(\tn{supp} \, f_i, \tn{supp} \, f_j) \geq \sizeofboundary \textnormal{ for all } i \neq j. \\
        \end{equation}

In this paper, we study a special case of \eqref{main_problems} where $K = 2$,  
and $\Omega = B_{R_2}(0) \setminus \overline{B_{R_1}(0)} = \{x \in \mathbb{R}^n : R_1 < \norm{x} < R_2\}$ is the annular domain in $\mathbb{R}^n$ bounded by two concentric spheres with radius $R_1$ and $R_2$ respectively (with $\sizeofboundary < R_1 < R_1 + \sizeofboundary \leq R_2$), so
\begin{equation}
\begin{split}
(\partial \Omega)_{\leq \sizeofboundary} := \{x \in \mathbb{R}^n \setminus \Omega: d(x, \partial \Omega) \leq \sizeofboundary\} 
\\ = \{x \in \mathbb{R}^n : R_1 - R \leq \norm{x} \leq R_1\} \cup \{x \in \mathbb{R}^n : R_2 \leq \norm{x} \leq R_2 + R\}.
\end{split}
\end{equation}  

For the boundary data, we assume that 

\begin{equation}
\begin{split}
    f_1(x) = \begin{cases} M_1, & R_1 -\sizeofboundary \leq \norm{x} \leq R_1,
    \\ 0, & R_2 \leq \norm{x} \leq R_2 + \sizeofboundary.
    \end{cases}
\end{split}    
\end{equation}

\begin{equation}
\begin{split}
    f_2(x) = \begin{cases} 0, & R_1 -\sizeofboundary \leq \norm{x} \leq  R_1,
    \\ M_2, & R_2 \leq \norm{x} \leq R_2 + \sizeofboundary, 
    \end{cases}
\end{split}    
\end{equation}
where $M_1$ and $M_2$ are two given positive constants.

Consequently, the system \eqref{main_problems} can be rewritten as 

\begin{equation}
\label{main_problems_rewritten}
\begin{cases} \Delta u^{\varepsilon}_{1}(x) = \frac{1}{\eps^2} u^{\eps}_1(x)   H_R(u^{\eps}_2)(x)\quad & \text{ in } \{x \in \mathbb{R}^n: R_1 < \norm{x} < R_2\},\\
\Delta u^{\varepsilon}_2(x) = \frac{1}{\eps^2} u^{\eps}_2(x)   H_R(u^{\eps}_1)(x)\quad & \text{ in } \{x \in \mathbb{R}^n: R_1 < \norm{x} < R_2\},\\
        u^{\eps}_1 = M_1 & \text{ on } \{x \in \mathbb{R}^n: R_1 -\sizeofboundary \leq \norm{x} \leq R_1\}, \\
        u^{\eps}_1 = 0 & \text{ on } \{x \in \mathbb{R}^n: R_2 \leq \norm{x} \leq R_2 + \sizeofboundary\}, \\
        u^{\eps}_2 = 0 & \text{ on } \{x \in \mathbb{R}^n: R_1 -\sizeofboundary \leq \norm{x} \leq R_1\}, \\
        u^{\eps}_2 = M_2 & \text{ on } \{x \in \mathbb{R}^n: R_2 \leq \norm{x} \leq R_2 + \sizeofboundary\},
       \end{cases}
\end{equation}

The existence of positive solutions $(u_1^\eps,\ldots,u_K^\eps)$ of the system \eqref{main_problems} was proved in \cite{CL2}. Uniqueness of solutions to system \eqref{main_problems} was proved in \cite{Bozorgnia}. It was also shown in \cite{CL2} that solutions converge to a limit configuration $(u_1,\ldots, u_K)$ as $\eps \to 0$, where the supports of the populations $u_i$  are mutually disjoint and separated from each other by distance $\sizeofboundary$. 
The regularity of the free boundary for $n=2$ was also established. For system \eqref{main_problems}, one of the main challenges in the analysis of the free boundaries is that classical techniques (i.e., monotonicity formulas) cannot be used. The techniques developed in \cite{CL2} rely on the concepts of {\it asymptotic cone} and {\it asymptotic angle}. At a singular point, the angle is measured as the intersection of asymptotic cones.  
This construction was generalized in \cite{ChPaTo26_2} to characterize the regular and singular points in any dimension in terms of angles and densities.

For the case of $n=2$, \cite{CL2} showed that the singular points are isolated and the regular set is open and locally $C^1$. It was also shown, under additional conditions, that the free boundary is Lipschitz. A result on the equality of angles is also established. Another tool for $n=2$ is the construction of harmonic functions (i.e., barrier functions) on cones vanishing on the boundary.

For higher dimensions, \cite{ChPaTo26_2} showed that if the angles at the singular points are bounded away from $\frac{n\omega_n}{2}$, the regular set is open and locally $C^1$. It was also shown that under a convexity condition, the free boundary consist of finitely many hyperplanes of dimension $n-1$, and the angles of the singular points are bounded above by $\frac{n\omega_n}{3}$. Moreover, if there are two free boundary points that are at distance $\sizeofboundary$ apart form each other, either they are both regular or they are both singular. A symmetry condition on the free boundary is also derived.

However, in higher dimensions the upper bound on the Hausdorff dimension of the singular set remains open. The only known contribution so far in this direction is \cite{ChPaTo26_3}, where it was shown that there are two concrete examples in which the singular set is nonempty with Hausdorff dimension exactly $n-2$. The proof is based on finite group action, the uniqueness of the solution, and the rotational invariance of the system \eqref{main_problems}.

System \eqref{main_problems} is an example of \eqref{modelproblem}, the Gause-Lotka-Volterra system in population dynamics, that models coexistence of species that live in the same territory, diffuse, and compete for limited resources. This can be modeled by the following equation:
\begin{equation}\label{modelproblem}L_i(u^\eps _i)=\frac{u^\eps_i}{\eps^2}F(u^\eps_1,\ldots,u^\eps_K)
\end{equation}
in some domain $\Omega$, where $u^\eps_i$ is a positive function representing the density of the $i$-th species, $L_i$ encodes  the diffusion of  $u^\eps_i$, and   $u_i^\eps F(u^\eps_1,\ldots,u^\eps_K)/\eps^2$ models  the  attrition of the species $i$ due to competition with the others. The interaction functional  $F$ is strictly positive whenever the supports of two or more species overlap.
The smaller the parameter $\eps$, the stronger the competition among species. 
In the limit as $\eps\to0^+$ the high  competition  forces the species  to segregate, meaning $u_iu_j=0$ for $j\neq i$.

Another example of system \eqref{modelproblem} is given by 
\begin{equation}
\label{adjacentsegregation}
\begin{cases} \Delta u^{\varepsilon}_{i}= \frac{1}{\eps^2}   \sum_{j \neq i}  u^{\eps}_iu^{\eps}_j\quad & \text{ in } \Omega,\\
        u^{\eps}_i= f_i & \text{ on } \partial \Omega, \\
        u^{\eps}_i \geq 0 & \text{ in } \overline{\Omega}.
\end{cases}
\end{equation}
The existence of positive solutions to \eqref{adjacentsegregation} was initially investigated by Dancer and Du \cite{DancerDu1, DancerDu2} in the case of three species. 
Convergence to a segregated limit configuration as  $\eps\to0^+$ was later proven by Dancer, Hilhorst, Mimura, and Peletier \cite{DancerHilMimPel}. 
 More general classes of linear competitive systems, including \eqref{adjacentsegregation} as a special case, have been studied by Conti, Terracini, and Verzini \cite{Conti4, Conti3, Conti5}. %
 We also refer to \cite{CL7, Conti6} for related optimal partition problems involving the first eigenvalue of the Laplace operator.  
 The geometric properties of the free boundaries $\partial\{u_i>0\}\cap\Omega$ for the system \eqref{adjacentsegregation} have been investigated by Caffarelli, Karakhanyan and Lin \cite{CL5} (see also \cite{CL4}). It was shown that the free boundary splits into two parts: a regular set,  which is a locally  analytic surface, and a singular set, which is a closed set of Hausdorff dimension at most $n-2$. Singular points occur where the boundaries of three or more connected components of the supports intersect. See \cite{TavarezTerracini} for similar results applied to a broader class of systems.

{

The system \eqref{adjacentsegregation}, when the Laplace operator is replaced by the fully nonlinear negative Pucci operator, has been studied by Quitalo in \cite{V} and Caffarelli, Quitalo, Patrizi, and Torres in {\cite{CL3}}. 


The interaction between the populations in \eqref{adjacentsegregation} is local, meaning that it depends only on the value of the densities, $u_i(x)$, at the point $x$. The segregation is adjacent since the supports of the populations have a common boundary. However, there are many processes where the growth of species $i$ is inhibited by populations $j$ occupying an entire neighborhood around $x$, see for example \cite{CuMaMa, MiEiFang}. Caffarelli, Patrizi, and Quitalo \cite{CL2} introduced system \eqref{main_problems} with the Laplace operator as an example to model non-local interactions. The same system was studied by the author, Patrizi, and Torres in \cite{ChPaTo26} with Laplacian replaced by the negative Pucci operator.  
When $H$ is given by \eqref{H1} with $w^2(y)$ in place of $w(y)$, minimizing solutions and the limiting configurations of \eqref{main_problems} have been studied in \cite{NS1,NS2}.

We state our first main result as follows.

\begin{theorem}
\label{Main_Result_1}

{

Consider the system \eqref{main_problems_rewritten}, where $\Omega = \{x \in \mathbb{R}^n: R_1 < \norm{x} < R_2\}$, $M_1 > 0$, $M_2 > 0$, $\sizeofboundary > 0$, $0 < \sizeofboundary < R_1 < R_2$, $R_1 + \sizeofboundary < R_2$, and the function $H_{\sizeofboundary}(w)$ is given by \eqref{H1}. Let $(u^{\epsilon}_1, u^{\epsilon}_2) \in C^2(\Omega; \mathbb{R}^2) \cap C^0(\overline{\Omega}; \mathbb{R}^2)$ be a solution of \eqref{main_problems_rewritten} for any $\epsilon > 0$, and suppose that $\{(u^{\epsilon}_1, u^{\epsilon}_2)\}_{\epsilon > 0}$ converges along a subsequence as $\epsilon \rightarrow 0^+$ to $(u_1, u_2)$. Then we have the followings:
\begin{enumerate}
\item There exists a unique $r^* \in [R_1, R_2 -\sizeofboundary], r^* = r^*(M_1, M_2, \sizeofboundary, R_1, R_2)$ such that 
\begin{equation}
\begin{split}
\partial \{u_1 > 0\} \cap \Omega = \{x \in \mathbb{R}^n: \norm{x} = r^*\}, \\
\partial \{u_2 > 0\} \cap \Omega = \{x \in \mathbb{R}^n: \norm{x} = r^* + \sizeofboundary\}.
\end{split}
\end{equation}
In particular, all the free boundary points are regular.
\item For $n = 2$, $r^*$ is the unique solution of the equation 
\begin{equation}
     \frac{\log (\frac{r^*}{R_1})}{\log (\frac{R_2}{r^* + \sizeofboundary})} - \frac{M_1}{M_2} = 0
\end{equation}
in the interval $[R_1, R_2]$.
\item For $n \geq 3$, $r^*$ is the unique solution of the equation 
\begin{equation}
    \frac{\frac{1}{R_1^{n-2}} - \frac{1}{(r^*)^{n-2}}}{\frac{1}{(r^* + \sizeofboundary)^{n-2}} - \frac{1}{R_2^{n-2}}} - \frac{M_1}{M_2} = 0
\end{equation}
in the interval $[R_1, R_2]$.
\item In any dimensions, for fixed $M_2, \sizeofboundary, R_1, R_2$, $r^* = r^*(M_1, M_2, \sizeofboundary, R_1, R_2)$ is increasing as a function of $M_1$, and
\begin{equation}
\begin{split}
\lim_{M_1 \rightarrow 0^+}r^*(M_1, M_2, \sizeofboundary, R_1, R_2) = R_1, \\
\lim_{M_1 \rightarrow \infty} r^*(M_1, M_2, \sizeofboundary, R_1, R_2) = R_2 -\sizeofboundary.    
\end{split}
\end{equation}
\item In any dimensions, for fixed $M_1, \sizeofboundary, R_1, R_2$, $r^* = r^*(M_1, M_2, \sizeofboundary, R_1, R_2)$ is decreasing as a function of $M_2$, and 
\begin{equation}
\begin{split}
\lim_{M_2 \rightarrow 0^+}r^*(M_1, M_2, \sizeofboundary, R_1, R_2) = R_2 - \sizeofboundary, \\
\lim_{M_2 \rightarrow \infty} r^*(M_1, M_2, \sizeofboundary, R_1, R_2) = R_1.  
\end{split}
\end{equation}
\item In any dimensions, for fixed $M_1, M_2, R_1, R_2$, $r^* = r^*(M_1, M_2, R, R_1, R_2)$ is decreasing as a function of $\sizeofboundary$, and
\begin{equation}
\begin{split}
\lim_{\sizeofboundary \rightarrow 0^+} r^*(M_1, M_2, \sizeofboundary, R_1, R_2) = r_0 = r_0(M_1, M_2, R_1, R_2),  
\end{split}
\end{equation}
where 
\begin{equation}
\begin{split}
    r_0 = r_0(M_1, M_2, R_1, R_2) = \begin{cases} R_1^{\frac{M_2}{M_1 + M_2}}R_2^{\frac{M_1}{M_1 + M_2}}, & n = 2, \\
    (\frac{M_1}{M_1 + M_2}\frac{1}{R_2^{n-2}} + \frac{M_2}{M_1 + M_2}\frac{1}{R_1^{n-2}})^{\frac{1}{2-n}}, & n \geq 3, 
    \end{cases}
\end{split}
\end{equation}
\end{enumerate}
}
\end{theorem}

\begin{remark}
\label{Remark_about_Main_Result_1}

{\rm

We elaborate more on part (6) of Theorem \ref{Main_Result_1}. Consider the local analogue (by setting $\sizeofboundary = 0$) of \eqref{main_problems_rewritten}. That is, we consider
\begin{equation}
\label{main_problems_local}
\begin{cases} \Delta u^{\varepsilon}_{1}(x) = \frac{1}{\eps^2} u^{\eps}_1(x)   u^{\eps}_2(x)\quad & \text{ in } \{x \in \mathbb{R}^n: R_1 < \norm{x} < R_2\},\\
\Delta u^{\varepsilon}_2(x) = \frac{1}{\eps^2} u^{\eps}_2(x)   u^{\eps}_1(x)\quad & \text{ in } \{x \in \mathbb{R}^n: R_1 < \norm{x} < R_2\},\\
        u^{\eps}_1 = M_1 & \text{ on } \{\norm{x} = R_1\}, \\
        u^{\eps}_1 = 0 & \text{ on } \{\norm{x} = R_2\}, \\
        u^{\eps}_2 = 0 & \text{ on } \{\norm{x} = R_1\}, \\
        u^{\eps}_2 = M_2 & \text{ on } \{\norm{x} = R_2\},
       \end{cases}
\end{equation}
Then the difference $u(x) = u_1(x) - u_2(x)$ satisfies 
\begin{equation}
\begin{split}
\begin{cases} \Delta u = 0 & \text{ in } \{x \in \mathbb{R}^n: R_1 < \norm{x} < R_2\}, \\
u = M_1 & \text{ on } \{\norm{x} = R_1\}, \\
u = -M_2 & \text{ on } \{\norm{x} = R_2\}, \\
\end{cases}
\end{split}
\end{equation}
It follows that $\partial \{u_1 > 0\} \cap \Omega = \partial \{u_2 > 0\} \cap \Omega = \{\norm{x} = r_0\}$, which is also the zero level set of the function $u = u_1 - u_2$. Therefore, in part (6) of Theorem \ref{Main_Result_1}, it was asserted that as $\sizeofboundary \rightarrow 0^+$, the free boundary for the long range case in \eqref{main_problems_rewritten} converges to the free boundary for the adjacent case \eqref{main_problems_local}. This connects the theory of long range interaction with that of the local interaction in this setting.  
}
\end{remark}

In this paper, we also study the more general case in which $R_1, R_2, M_1, M_2$, and $\sizeofboundary$ depends on a parameter $t$, where $t$ ranges over an open interval $I$ in the real line. Therefore, we consider the following problem
\begin{equation}
\label{main_problems_rewritten_one_parameter_family}
\begin{cases} \Delta u^{\varepsilon}_{1,t}(x) = \frac{1}{\eps^2} u^{\eps}_{1,t}(x)   H_{\sizeofboundary_t}(u^{\eps}_{2,t})(x)\quad & \text{ in } \{x \in \mathbb{R}^n: R_{1,t} < \norm{x} < R_{2,t}\},\\
\Delta u^{\varepsilon}_{2,t}(x) = \frac{1}{\eps^2} u^{\eps}_{2,t}(x)   H_{\sizeofboundary_t}(u^{\eps}_{1,t})(x)\quad & \text{ in } \{x \in \mathbb{R}^n: R_{1,t} < \norm{x} < R_{2,t}\},\\ 
        u^{\eps}_{1,t}(x) = M_{1,t}& \text{ on } \{x \in \mathbb{R}^n: R_{1,t} -\sizeofboundary_t \leq \norm{x} \leq R_{1,t}\}, \\
        u^{\eps}_{1,t}(x) = 0 & \text{ on } \{x \in \mathbb{R}^n: R_{2,t} \leq \norm{x} \leq R_{2,t} + \sizeofboundary_t\}, \\
        u^{\eps}_{2,t}(x) = 0 & \text{ on } \{x \in \mathbb{R}^n: R_{1,t} -\sizeofboundary_t \leq \norm{x} \leq R_{1,t}\}, \\
        u^{\eps}_{2,t}(x) = M_{2,t} & \text{ on } \{x \in \mathbb{R}^n: R_{2,t} \leq \norm{x} \leq R_{2,t} + \sizeofboundary_t\},
       \end{cases} 
\end{equation}
and we would like to understand how the free boundary changes as $t$ varies. We first consider the case in which $n = 2$. 

The following is the second main result of the paper.

\begin{theorem}
\label{Main_Result_2}

Assume $n = 2$. Consider the one-parameter family of elliptic system \eqref{main_problems_rewritten_one_parameter_family} indexed by a parameter $t \in I$, where $I$ is an open interval in the real line, the function $H_{\sizeofboundary}(w)$ is given by \eqref{H1}, $0 < \sizeofboundary_t < R_{1,t} < R_{2,t}$, $R_{1,t} + \sizeofboundary_t < R_{2,t}$ for all $t \in I$. Take $\Omega_t := \{R_{1,t} < \norm{x} < R_{2,t}\}$ for all $t \in I$. 

Let $(u^{\epsilon}_{1,t}, u^{\epsilon}_{2,t}) \in C^2(\Omega_t; \mathbb{R}^2) \cap C^0(\overline{\Omega_t}; \mathbb{R}^2)$ be a solution of \eqref{main_problems_rewritten_one_parameter_family} for any $\epsilon > 0, t \in I$, and suppose that $\{(u^{\epsilon}_{1,t}, u^{\epsilon}_{2,t})\}_{\epsilon > 0}$ converges along a subsequence as $\epsilon \rightarrow 0^+$ to $(u_{1,t}, u_{2,t})$. We also assume that the mapping $I \rightarrow \mathbb{R}$ given by $t \mapsto R_{1,t}$, $t \mapsto R_{2,t}$, $t \mapsto \sizeofboundary_{t}$, $t \mapsto {M}_{1,t}$, $t \mapsto {M}_{2,t}$ are all differentiable on $I$. For any $t \in I$, let $r^* = r^*_{t} \in [R_{1,t}, R_{2,t} - \sizeofboundary_t]$ be the unique number as guaranteed in Theorem \ref{Main_Result_1} (1) such that 
\begin{equation}
\begin{split}
\partial \{u_{1,t} > 0\} \cap \Omega_t = \{\norm{x} = r^*_{t}\},
\\ \partial \{u_{2,t} > 0\} \cap \Omega_t = \{\norm{x} = r^*_{t} + \sizeofboundary_t\}.
\end{split}
\end{equation}
Then we have the followings.
\begin{enumerate}
\item The mapping $I \rightarrow \mathbb{R}$ given by $t \mapsto r^*_t$ is also differentiable on $I$, with 
\begin{equation}
    \left(\frac{d}{dt}r^*_t\right)\left(\frac{1}{x_{2,t}r^*_t} + \frac{x_{1,t}}{x_{2,t}^2(r^*_t + \sizeofboundary_t)}\right) = \frac{d}{dt}\left(\frac{M_{1,t}}{M_{2,t}}\right) + \frac{\frac{d}{dt}R_{1,t}}{R_{1,t}}\frac{1}{x_{2,t}} + \frac{x_{1,t}}{x_{2,t}^2}\frac{\frac{d}{dt}R_{2,t}}{R_{2,t}} - \frac{x_{1,t}\frac{d}{dt}\sizeofboundary_t}{x_{2,t}^2(r^*_t + \sizeofboundary_t)},
\end{equation}
where 
\begin{equation}
    x_{1,t} = \log(\frac{r^*_t}{R_{1,t}}) > 0, x_{2,t} = \log(\frac{R_{2,t}}{r^*_{t} + \sizeofboundary_t}) > 0
\end{equation}
for all $t \in I$.
\item Assume that $R_{1,t}, R_{2,t}$, and $\sizeofboundary_t$ are all constants in $t \in I$. Then $r^*_t$ is a constant in $t \in I$ if $\frac{M_{1,t}}{M_{2,t}}$ is constant in $t \in I$.  
\item Assume $\frac{M_{1,t}}{M_{2,t}}$ is a constant $\alpha$ in $t \in I$ and $\sizeofboundary_t$ is constant in $t \in I$. Then 
$r^*_t$ is constant in $t \in I$ if 
\begin{equation}
    \frac{\frac{d}{dt} (\log R_{1,t})}{\frac{d}{dt} (\log R_{2,t})} = -\alpha
\end{equation}
for all $t \in I$.
\end{enumerate}
\end{theorem}

For $n \geq 3$, we have the following result instead.

\begin{theorem}
\label{Main_Result_3}

Assume $n \geq 3$. Consider the one-parameter family of elliptic system \eqref{main_problems_rewritten_one_parameter_family} indexed by a parameter $t \in I$, where $I$ is an open interval in the real line, the function $H_{\sizeofboundary}(w)$ is given by \eqref{H1}, $0 < \sizeofboundary_t < R_{1,t} < R_{2,t}$, $R_{1,t} + \sizeofboundary_t < R_{2,t}$ for all $t \in I$. Take $\Omega_t := \{R_{1,t} < \norm{x} < R_{2,t}\}$ for all $t \in I$.

Let $(u^{\epsilon}_{1,t}, u^{\epsilon}_{2,t}) \in C^2(\Omega_t; \mathbb{R}^2) \cap C^0(\overline{\Omega_t}; \mathbb{R}^2)$ be a solution of \eqref{main_problems_rewritten_one_parameter_family} for any $\epsilon > 0, t \in I$, and suppose that $\{(u^{\epsilon}_{1,t}, u^{\epsilon}_{2,t})\}_{\epsilon > 0}$ converges along a subsequence as $\epsilon \rightarrow 0^+$ to $(u_{1,t}, u_{2,t})$. We also assume that the mapping $I \rightarrow \mathbb{R}$ given by $t \mapsto R_{1,t}$, $t \mapsto R_{2,t}$, $t \mapsto \sizeofboundary_{t}$, $t \mapsto {M}_{1,t}$, $t \mapsto {M}_{2,t}$ are all differentiable on $I$. For any $t \in I$, let $r^* = r^*_{t} \in [R_{1,t}, R_{2,t} - \sizeofboundary_t]$ be the unique number as guaranteed in Theorem \ref{Main_Result_1} (1) such that 
\begin{equation}
\begin{split}
\partial \{u_{1,t} > 0\} \cap \Omega_t = \{\norm{x} = r^*_{t}\},
\\ \partial \{u_{2,t} > 0\} \cap \Omega_t = \{\norm{x} = r^*_{t} + \sizeofboundary_t\}.
\end{split}
\end{equation}
Then we have the followings.
\begin{enumerate}
\item The mapping $I \rightarrow \mathbb{R}$ given by $t \mapsto r^*_t$ is also differentiable on $I$, with 
\begin{equation}
\begin{split}
\left(\frac{dr^*_t}{dt}\right)(n-2)[\frac{1}{x_{2,t}}(r^*_t)^{1-n} + \frac{x_{1,t}}{x_{2,t}^2}(r^*_t + \sizeofboundary_t)^{1-n}] = \frac{d}{dt}\left(\frac{M_{1,t}}{M_{2,t}}\right) 
\\+ (n-2)\frac{1}{x_{2,t}}R^{1-n}_{1,t}\left(\frac{dR_{1,t}}{dt}\right) + (n-2)\frac{x_{1,t}}{x_{2,t}^2}R^{1-n}_{2,t}\left(\frac{dR_{2,t}}{dt}\right) - (n-2)\frac{x_{1,t}}{x_{2,t}^2}(r^*_t + \sizeofboundary_t)^{1-n}\left(\frac{d\sizeofboundary_t}{dt}\right),
\end{split}
\end{equation}
where 
\begin{equation}
    x_{1,t} = R_{1,t}^{2-n} - (r^*_t)^{2-n} > 0, x_{2,t} = (r^*_t + \sizeofboundary_t)^{2-n} - R_{2,t}^{2-n} > 0
\end{equation}
for all $t \in I$.
\item Assume that $R_{1,t}, R_{2,t}$, and $\sizeofboundary_t$ are all constants in $t \in I$. Then $r^*_t$ is a constant in $t \in I$ if $\frac{M_{1,t}}{M_{2,t}}$ is constant in $t \in I$.  
\item Assume $\frac{M_{1,t}}{M_{2,t}}$ is a constant $\alpha$ in $t \in I$ and $\sizeofboundary_t$ is constant in $t \in I$. Then 
$r^*_t$ is constant in $t \in I$ if 
\begin{equation}
    \frac{\frac{d}{dt}(R_{1,t}^{2-n})}{\frac{d}{dt}(R_{2,t}^{2-n})} = -\alpha
\end{equation}
for all $t \in I$.
\end{enumerate}
\end{theorem}

\section{Preliminaries}
We first recall some known definitions and results in \cite{Bozorgnia}, \cite{CL2}, \cite{ChPaTo26}, \cite{ChPaTo26_2}, and \cite{ChPaTo26_3} for the convenience of the reader. 

Firstly, we have the existence and uniqueness of solutions of \eqref{main_problems}. The existence of solutions was proved in \cite[Thm. 4.1]{CL2} using Perron's method and Schauder fixed point theorem. The similar approach was adopted in \cite[Thm. 1]{ChPaTo26} to extend the result with negative Pucci operator as diffusion. The uniqueness result was proved in \cite[Thm. 3.3]{Bozorgnia} by sub- and sup-solution method. 

\begin{theorem}
\label{existence_and_uniqueness}
\cite[Theorem 3.3]{Bozorgnia}\cite[Theorem 4.1]{CL2}\cite[Theorem 1]{ChPaTo26}

{
Let $\Omega \subset \mathbb{R}^n$ be a bounded Lipschitz domain, and assume \eqref{fiassumption}-\eqref{f_i-f_j_disjointsupport} hold true. Consider the elliptic system \eqref{main_problems} with $H_R(w)(x)$ given by \eqref{H1}. Then for any $\eps > 0$, and $0 < \sizeofboundary \leq 1$, there exists a unique positive solution 
$(u^{\eps}_1,\ldots, u^{\eps}_K) \in  C^{\alpha}(\overline{\Omega};\mathbb{R}^K)\cap C_{loc}^{2,\alpha}(\Omega;\mathbb{R}^K)$ of \eqref{main_problems}, for some $0 < \alpha < 1$.
}
\end{theorem}

It was also known that there is a subsequential limit of $\{(u^{\epsilon}_1, \cdots, u^{\epsilon}_K)\}_{\epsilon > 0}$ as $\epsilon \rightarrow 0^+$.

\begin{theorem}
\cite[Corollary 5.6]{CL2}\cite[Theorem 2]{ChPaTo26}
\label{main theorem3-old}

{

Let $\Omega \subset \mathbb{R}^n$ be a bounded Lipschitz domain, and assume \eqref{fiassumption}-\eqref{f_i-f_j_disjointsupport} hold true. Consider the elliptic system \eqref{main_problems} with $H_R(w)(x)$ given by \eqref{H1}.
For any $\eps > 0$ and $0<R\leq 1$,  let $(u^{\eps}_1, \ldots, u^{\eps}_K)$ be a solution to \eqref{main_problems}. Then there exists a subsequence of $\{(u^{\eps}_1,\ldots, u^{\eps}_K) \}_{\eps > 0}$
that converges locally uniformly in $\Omega$ to a limit function $(u_1, \ldots, u_K)$ as $\eps \rightarrow 0^+$. Moreover, the limit function $(u_1, \ldots, u_K)$
has the following properties:
\begin{enumerate}
\item Each function $u_i$ is locally Lipschitz continuous on $\Omega$.
\item  $\Delta u_i = 0$ on $\{u_i > 0\}\cap\Omega$,  for any $i=1,\ldots, K$.
\item  For any $1 \leq i < j \leq K$, the supports of the function $u_i$ and $u_j$ are at distance at least $\sizeofboundary$ from each other. 
\end{enumerate}
}
\end{theorem}

\begin{definition}

{\rm The sets $\partial \{u_i > 0\} \cap \Omega, i = 1, \cdots, K$ are called the free boundaries.}
\end{definition}

In addition, some preliminary properties of the free boundaries were known.

\begin{theorem}
\cite[Corollary 6.2, 6.5]{CL2}\cite[Theorem 7.1]{CL2}\cite[Theorem 3, 4]{ChPaTo26}
\label{main theorem5-old}

{
If $x_0 \in \partial \{ u_i > 0\}\cap\Omega$ for some $i = 1, \ldots, K$, then there is an exterior tangent ball $B_R(y)$ at $x_0$, and the set $\{u_i > 0\}\cap\Omega$ has finite perimeter. Furthermore, if $x_0 \in \partial \{u_i > 0\} \cap \Omega$, there exists $j \neq i$ such that $\overline{B_{\sizeofboundary}(x_0)} \cap \partial \{u_j > 0\} \neq \emptyset$.  
}
\end{theorem}

We recall the definitions of regular and singular free boundary points as follows.

\begin{definition}
\label{basic_definition}
\cite{CL2}\cite[Definition 2.1, 2.2]{ChPaTo26_2}

{\rm For each $i = 1, \cdots, K$, we define $S_i:= \{x \in \Omega : u_i (x) > 0 \}$, and $C_i := \cup_{j \neq i} S_j$.
A free boundary point $x_0 \in \partial S_i\cap \Omega$ is said to be {\bf regular} if there exists a unique $x_1 \in \partial C_i \cap \Omega$ for which $d(x_0,x_1) = \sizeofboundary.$ A free boundary point is said to be {\bf singular} if it is not regular. }
\end{definition}

Hence, $x_0 \in \partial S_i \cap \Omega$ is singular if and only if there are two distinct points $x_1, x_2 \in \partial C_i \cap \Omega, x_1 \neq x_2$ such that $d(x_0,x_1) = d(x_0,x_2) = \sizeofboundary.$

Another result we need is that solutions of the elliptic system \eqref{main_problems} are  invariant under orthogonal transformations.

Let $U$ be a bounded Lipschitz domain in $\mathbb{R}^n$. For $A \in O(n;\mathbb{R})$, we let $A(U) := \{Ax: x \in U\}$. (Here $O(n;\mathbb{R})$ denotes the group of all $n \times n$ real orthogonal matrices).

\begin{lemma}
\label{orthogonal_transformation}
\cite[Lemma 3.2]{ChPaTo26_3}

{

Let $U \subset \mathbb{R}^n$ be a bounded Lipschitz domain, and let $f_1, f_2,\cdots, f_K : (\partial U)_{\leq R} \rightarrow \mathbb{R}$ satisfy \eqref{fiassumption}-\eqref{f_i-f_j_disjointsupport}. Suppose that $(u^{\epsilon}_1,u^{\epsilon}_2,\cdots,u^{\epsilon}_K) \in C^2(U) \cap C^0(\overline{U})$ is a solution of 
\begin{equation}
\begin{cases} \Delta u^{\varepsilon}_{i}= \frac{1}{\eps^2} u^{\eps}_i  \sum_{j \neq i}  H_R(u^{\eps}_j)(x)\quad & \text{ in } U,\\
        u^{\eps}_i= f_i & \text{ on } (\partial U)_{\leq {\sizeofboundary}}, \\
        
        u^{\eps}_i \geq 0 & \text{ in } U \cup (\partial U)_{\leq \sizeofboundary},
       \end{cases}
\end{equation}
with $H_{\sizeofboundary}(w)(x)$ given by \eqref{H1}.
Let $A \in O(n;\mathbb{R})$ be such that $A(U) = U$. Take $w_i^{\epsilon} := u^{\epsilon}_i(Ax)$ for all $\epsilon > 0, i = 1,\cdots,K$, and $x \in \Omega$. Then  
\begin{equation}
\begin{cases} \Delta w^{\varepsilon}_{i} = \frac{1}{\eps^2} w^{\eps}_i  \sum_{j \neq i}  H_R(w^{\eps}_j)(x)\quad & \text{ in } U,\\
w^{\eps}_i = f_i \circ A & \text{ on } (\partial U)_{\leq {\sizeofboundary}}, \\
w^{\eps}_i \geq 0 & \text{ in } U \cup (\partial U)_{\leq \sizeofboundary},
       \end{cases}
\end{equation}
}
\end{lemma}

In this paper, we also need the following free boundary condition obtained in \cite{CL2} in the analysis of the free boundary later. We state the result here for the convenience of the reader.

\begin{theorem}
\label{The_Free_Boundary_Condition}
\cite[Theorem 9.2]{CL2}

{
In any dimension, assume that we have two populations, $H$ is defined as in \eqref{H1}, $0 \in \partial \{u_1 > 0\} \cap \Omega, \sizeofboundary e_n \in \partial \{u_2 > 0\} \cap \Omega$, and $\partial \{u_1 > 0\} \cap \Omega$ and $\partial \{u_2 > 0\} \cap \Omega$ are of class $C^2$ in a neighborhood of $0$ and $\sizeofboundary e_n$, respectively. Let $\chi_i(0)$ denote the principal curvatures of $\partial \{u_1 > 0\} \cap \Omega$ at $0$ where the outward is the positive direction, and let $\chi_i(\sizeofboundary e_n)$ denote the principal curvatures of $\partial \{u_2 > 0\} \cap \Omega$ at $0$ where the inward is the positive direction. Assume that $\chi_i(0) < \frac{1}{\sizeofboundary}$ and $\chi_{i}(\sizeofboundary e_n) < \frac{1}{\sizeofboundary}$.  Then we have the following relation between the normal derivatives of $u_1$ and $u_2$:
\begin{equation}
\label{The_First_Free_Boundary_Condition}
\frac{u^1_{\nu}(0)}{u^2_{\nu}(\sizeofboundary e_n)} = \prod_{i = 1, \chi_i(0) \neq 0}^{n-1}\frac{\chi_i(0)}{\chi_i(\sizeofboundary e_n)}, \textnormal{ if } \chi_i(0) \neq 0 \textnormal{ for some } i = 1, \cdots n-1, 
\end{equation}
\begin{equation}
\frac{u^1_{\nu}(0)}{u^2_{\nu}(\sizeofboundary e_n)} = 1 \textnormal{ if } \chi_i(0) = 0 \textnormal{ for all } i = 1, \cdots n-1    
\end{equation}
}
\end{theorem}

\section{The Analysis of The Radius}

The section is dedicated to the proof of Theorem \ref{Main_Result_1}. Throughout this section, we consider the problem \eqref{main_problems_rewritten}, where $\Omega = \{x \in \mathbb{R}^n: R_1 < \norm{x} < R_2\}$, $M_1 > 0$, $M_2 > 0$, $0 < \sizeofboundary < R_1 < R_2$, $R_1 + \sizeofboundary < R_2$, and the function $H_{\sizeofboundary}(w)$ is given by \eqref{H1}.

We first make use of the uniqueness of the solution as asserted in Theorem \ref{existence_and_uniqueness} to show that under the conditions of Theorem \ref{Main_Result_1}, the free boundary consists of spheres.

\begin{lemma}
\label{Existence_of_r^*}

{
Under the above conditions, there is a $r^* \in [R_1, R_2 -\sizeofboundary]$ such that 
\begin{equation}
\begin{split}
\partial \{u_1 > 0\} \cap \Omega = \{x \in \mathbb{R}^n: \norm{x} = r^*\}, \\
\partial \{u_2 > 0\} \cap \Omega = \{x \in \mathbb{R}^n: \norm{x} = r^* + \sizeofboundary\}.
\end{split}
\end{equation}
Moreover, all the free boundary points are regular.
}
\end{lemma}
\begin{proof}

Assume that $(u^{\epsilon}_1, u^{\epsilon}_2)$ is a solution of \eqref{main_problems_rewritten}, and $A \in O(n; \mathbb{R})$. (That is, $A$ is a real orthogonal $n \times n$ matrix). Take $w^{\epsilon}_1(x) = u^{\epsilon}_1(Ax)$ and $w^{\epsilon}_2(x) = u^{\epsilon}_2(Ax)$ for all $x \in \Omega$. By Lemma \ref{orthogonal_transformation}, $(w^{\epsilon}_1, w^{\epsilon}_2)$ is also a solution of \eqref{main_problems_rewritten}. By the uniqueness part of Theorem \ref{existence_and_uniqueness}, we have $(w^{\epsilon}_1, w^{\epsilon}_2) = (u^{\epsilon}_1, u^{\epsilon}_2)$. By passing to a subsequential limit as $\epsilon \rightarrow 0^+$, we have $(u_1(x), u_2(x)) = (w_1(x), w_2(x)) = (u_1(Ax), u_2(Ax))$ for all $x \in \Omega$. Therefore, $u_1(x)$ and $u_2(x)$ are both radial functions, so $\partial \{u_1 > 0\} \cap \Omega = \{\norm{x} = r^*\}$ for some $r^*$. 
Also, by Theorem \ref{main theorem5-old} and the fact that 
\begin{equation}
\begin{cases} 
        u^{\eps}_1 = M_1 & \text{ on } \{x \in \mathbb{R}^n: R_1 -\sizeofboundary \leq \norm{x} \leq R_1\}, \\
        u^{\eps}_1 = 0 & \text{ on } \{x \in \mathbb{R}^n: R_2 \leq \norm{x} \leq R_2 + \sizeofboundary\}, \\
        u^{\eps}_2 = 0 & \text{ on } \{x \in \mathbb{R}^n: R_1 -\sizeofboundary \leq \norm{x} \leq R_1\}, \\
        u^{\eps}_2 = M_2 & \text{ on } \{x \in \mathbb{R}^n: R_2 \leq \norm{x} \leq R_2 + \sizeofboundary\},
       \end{cases}
\end{equation}
it follows that $\partial \{u_2 > 0\} \cap \Omega = \{\norm{x} = \sizeofboundary + r^*\}$. Since $R_1 \leq r^* \leq r^* + \sizeofboundary \leq R_2$, it follows that $r^* \in [R_1, R_2 -\sizeofboundary]$.  

Finally, all the free boundary points are regular by Definition \ref{basic_definition}. 
\end{proof}

Our next goal is to show that $r^*$ is uniquely determined by $M_1, M_2, R_1, R_1,$ and $\sizeofboundary$, and find $r^*$ in terms of the five variables $M_1, M_2, R_1, R_1,$ and $\sizeofboundary$. Before we proceed, we first require a notation and a Calculus fact.

\begin{notation}
\label{Notation_for_Two_Standard_Functions}

Assume that $R_1, R_2 > 0$, and $\sizeofboundary$ are such that $\sizeofboundary < R_1 < R_1 + \sizeofboundary < R_2$.
\begin{enumerate}
\item We define $\phi_{R_1, R_2, \sizeofboundary, 2}: [R_1, R_2 -\sizeofboundary) \rightarrow \mathbb{R}$ by setting 
\begin{equation}
    \phi_{R_1, R_2, \sizeofboundary, 2}(x) : = \frac{\log(\frac{x}{R_1})}{\log(\frac{R_2}{x + \sizeofboundary})}
\end{equation}
for all $x \in [R_1, R_2 - \sizeofboundary)$.
\item For all $n \geq 3$, we define $\phi_{R_1, R_2, \sizeofboundary, n}: [R_1, R_2 -\sizeofboundary) \rightarrow \mathbb{R}$ by setting 
\begin{equation}
    \phi_{R_1, R_2, \sizeofboundary, n}(x) : = \frac{R_1^{2-n} - x^{2-n}}{(x + \sizeofboundary)^{2 - n} - R_2^{2-n}}
\end{equation}
for all $x \in [R_1, R_2 - \sizeofboundary)$.
\end{enumerate}
\end{notation}

We need the following result regarding $\phi_{R_1, R_2, \sizeofboundary, n}(x)$, whose proof is standard.

\begin{lemma}
\label{Phi_is_Strictly_Increasing}
Consider the same setting as in Notation \ref{Notation_for_Two_Standard_Functions}. Then we have the followings:
\begin{enumerate}
\item For any $n \geq 2$, the function $\phi_{R_1, R_2, \sizeofboundary, n}(x)$ is a non-negative continuous strictly increasing function on $[R_1, R_2 - \sizeofboundary)$, with $\lim_{x \rightarrow (R_2 - \sizeofboundary)^-}\phi_{R_1, R_2, \sizeofboundary, n}(x) = \infty$ and $\phi_{R_1, R_2, \sizeofboundary, n}(R_1) = 0$.
\item If $0 < r_2 < r_1 < R_2 - R_1$, we have $\phi^{-1}_{R_1, R_2, 0, n}(x) \geq \phi^{-1}_{R_1, R_2, r_2, n}(x) \geq \phi^{-1}_{R_1, R_2, r_1, n}(x)$ for all $x \in [0,\infty)$.
\item We have $\lim_{\sizeofboundary \rightarrow 0^+}\phi^{-1}_{R_1, R_2, \sizeofboundary,n}(x) = \phi^{-1}_{R_1, R_2, 0,n}(x)$ for all $x \in [0,\infty)$.
\end{enumerate}
\end{lemma}
\begin{proof}
(1) is clear, and (2) follows easily from the definition of $\phi_{R_1, R_2, \sizeofboundary, n}(x)$ in Notation \ref{Notation_for_Two_Standard_Functions}. Therefore, we focus on proving (3) only. 
It suffices to show that if $\{r_k\}_{k = 1}^{\infty} \subset [0,\infty)$ is a sequence with $r_1 \geq r_2 \geq r_3 \geq \cdots$ and $r_k \rightarrow 0$, we have $\lim_{k \rightarrow \infty}\phi^{-1}_{R_1, R_2, r_k,n}(x) = \phi^{-1}_{R_1, R_2, 0,n}(x)$ for all $x \in [0, \infty)$. Take $y_k := \phi^{-1}_{R_1, R_2, r_k, n}(x)$, so by (2), $\{y_k\}_{k \in \mathbb{N}}$ is an increasing sequence, so $\{y_k\}_{k \in \mathbb{N}}$ converges to $\sup_{k \in \mathbb{N}} y_k =: y_0 \leq R_2$. Since $\phi_{R_1, R_2, r_k, n}(y_k) = x$, we have
\begin{equation}
\label{The_Equation_Satisfied_by_y_k}
\begin{split}
\begin{cases}
\frac{\log(\frac{y_k}{R_1})}{\log(\frac{R_2}{y_k + r_k})} = x, & n = 2, \\
\frac{R_1^{2-n} - y_k^{2-n}}{(y_k + r_k)^{2 - n} - R_2^{2-n}} = x, & n \geq 3.
\end{cases}
\end{split}
\end{equation}
Pass to the limit as $k \rightarrow \infty$ in \eqref{The_Equation_Satisfied_by_y_k}, and we obtain 
\begin{equation}
\label{The_Equation_Satisfied_by_y_k}
\begin{split}
\begin{cases}
\frac{\log(\frac{y_0}{R_1})}{\log(\frac{R_2}{y_0})} = x, & n = 2, \\
\frac{R_1^{2-n} - y_0^{2-n}}{(y_0)^{2 - n} - R_2^{2-n}} = x, & n \geq 3.
\end{cases}
\end{split}
\end{equation}
This implies that $\phi_{R_1, R_2, 0, n}(y_0) = x$, so $y_0 = \phi_{R_1, R_2, 0, n}^{-1}(x)$. Consequently, $\lim_{k \rightarrow \infty}\phi^{-1}_{R_1, R_2, r_k,n}(x) = \phi^{-1}_{R_1, R_2, 0,n}(x)$, as desired.

\end{proof}

We are now ready to determine $r^*$ in the two-dimensional case.

\begin{lemma}
\label{Analysis_of_r^*_in_2D}

Consider the case in which $n = 2$. Then $r^* \in [R_1, R_2 - \sizeofboundary]$ is the unique solution to the equation
\begin{equation}
\label{The_Equation_Satisfied_by_r^*_in_2D}
\frac{\log(\frac{r^*}{R_1})}{\log (\frac{R_2}{r^* + \sizeofboundary})} - \frac{M_1}{M_2} = 0
\end{equation}
in $[R_1, R_2 - \sizeofboundary]$. 
\end{lemma}
\begin{proof}

Take 
\begin{equation}
\begin{split}
A = \{x \in \mathbb{R}^n: R_1 < \norm{x} < r^*\}, \\
B = \{x \in \mathbb{R}^n: r^* + \sizeofboundary < \norm{x} < R_2\}.
\end{split}
\end{equation}

Observe that $u_1(x)$ and the function $\frac{M_1}{\log (\frac{R_1}{r^*})}\log |x| - \frac{M_1\log (r^*)}{\log (\frac{R_1}{r^*})}$ are both harmonic in $A$ and take on the same boundary value on $\partial A$. (To see that they agree on $\partial A$, we note that $\partial A = \{\norm{x} = r^*\} \cup \{\norm{x} = R_1\}$, and they both vanish on $\{\norm{x} = r^*\}$, and they both take on the value $M_1$ on $\{\norm{x} = R_1\}$.) By the standard comparison principle for harmonic functions, we have $u_1(x) = \frac{M_1}{\log (\frac{R_1}{r^*})}\log |x| - \frac{M_1\log (r^*)}{\log (\frac{R_1}{r^*})}$ identically in $A$. 

Now take $\{e_1, e_2\} \subset \mathbb{R}^2$ to be the standard orthonormal basis of $\mathbb{R}^2$, so by a direct computation we have 
\begin{equation}
\label{Normal_Derivative_for_u_1}
\frac{\partial u_1}{\partial \nu}(r^*e_2) = \frac{M_1}{\log (\frac{R_1}{r^*})}\frac{1}{r^*},
\end{equation}
where the outward is the positive direction. 

By a similar argument applied to $u_2(x)$ and $\frac{M_2}{\log (\frac{R_2}{r^* + \sizeofboundary})}\log |x| - \frac{M_2\log (r^* + \sizeofboundary)}{\log (\frac{R_2}{r^* + \sizeofboundary})}$ in the annular domain $B$, we have $u_2(x) = \frac{M_2}{\log (\frac{R_2}{r^* + \sizeofboundary})}\log |x| - \frac{M_2\log (r^* + \sizeofboundary)}{\log (\frac{R_2}{r^* + \sizeofboundary})}$ identically in $B$. Similarly as in the case of $u_1$, we have 
\begin{equation}
\label{Normal_Derivative_for_u_2}
\frac{\partial u_2}{\partial \nu}((r^*+\sizeofboundary)e_2) = -\frac{M_2}{\log (\frac{R_2}{r^* + \sizeofboundary})}\frac{1}{r^* + \sizeofboundary},
\end{equation}
where we have a minus sign as now the inward is the positive direction.

Now by \eqref{Normal_Derivative_for_u_1}, \eqref{Normal_Derivative_for_u_2}, and \eqref{The_First_Free_Boundary_Condition} in Theorem \ref{The_Free_Boundary_Condition}, we have 
\begin{equation}
\frac{\frac{M_1}{\log (\frac{R_1}{r^*})}\frac{1}{r^*}}{-\frac{M_2}{\log (\frac{R_2}{r^* + \sizeofboundary})}\frac{1}{r^* + \sizeofboundary}} = \frac{\frac{1}{r^*}}{\frac{1}{r^* +\sizeofboundary}}.
\end{equation}
Therefore, 
\begin{equation}
\frac{M_1}{\log (\frac{R_1}{r^*})} = -\frac{M_2}{\log (\frac{R_2}{r^* + \sizeofboundary})}.
\end{equation}
Rearrange the equation, and we obtain \eqref{The_Equation_Satisfied_by_r^*_in_2D}. 

Finally, we observed that $r^*$ is the solution of the equation $\phi_{R_1, R_2, \sizeofboundary, 2}(r^*) = \frac{M_1}{M_2}$, so in view of Lemma \ref{Phi_is_Strictly_Increasing}, $r^*$ is unique.  
\end{proof}

Now using a similar technique, we can also determine $r^*$ in dimensions $n \geq 3$.

\begin{lemma}
\label{Analysis_of_r^*_in_nD}
Consider the case in which $n \geq 3$. Then $r^* \in [R_1, R_2 - \sizeofboundary]$ is the unique solution to the equation
\begin{equation}
\label{The_Equation_Satisfied_by_r^*_in_nD}
\frac{R_1^{2-n} - (r^*)^{2-n}}{{(r^* + \sizeofboundary)^{2-n}} - {R_2^{2-n}}} - \frac{M_1}{M_2} = 0
\end{equation}
in $[R_1, R_2 - \sizeofboundary]$. 
\end{lemma}
\begin{proof}

Take 
\begin{equation}
\begin{split}
A = \{x \in \mathbb{R}^n: R_1 < \norm{x} < r^*\}, \\
B = \{x \in \mathbb{R}^n: r^* + \sizeofboundary < \norm{x} < R_2\}.
\end{split}
\end{equation}

Observe that $u_1(x)$ and the function $\frac{M_1}{R_1^{2-n} - (r^*)^{2-n}}\norm{x}^{2-n} - \frac{(r^*)^{2-n}M_1}{R_1^{2-n} - (r^*)^{2-n}}$ are both harmonic in $A$ and take on the same boundary value on $\partial A$. (To see that they agree on $\partial A$, we note that $\partial A = \{\norm{x} = r^*\} \cup \{\norm{x} = R_1\}$, and they both vanish on $\{\norm{x} = r^*\}$, and they both take on the value $M_1$ on $\{\norm{x} = R_1\}$.) By the standard comparison principle for harmonic functions, we have $u_1(x) = \frac{M_1}{R_1^{2-n} - (r^*)^{2-n}}\norm{x}^{2-n} - \frac{(r^*)^{2-n}M_1}{R_1^{2-n} - (r^*)^{2-n}}$ identically in $A$. 

Now take $\{e_1, \cdots, e_n\} \subset \mathbb{R}^n$ to be the standard orthonormal basis of $\mathbb{R}^n$, so by a direct computation we have 
\begin{equation}
\label{Normal_Derivative_for_u_1_n_geq_3}
\frac{\partial u_1}{\partial \nu}(r^*e_n) = \frac{M_1}{R_1^{2-n} - (r^*)^{2-n}}(2-n)(r^*)^{1-n},
\end{equation}
where the outward is the positive direction.

By a similar argument applied to $u_2(x)$ and $\frac{M_2}{R_2^{2-n} - (r^*+\sizeofboundary)^{2-n}}\norm{x}^{2-n} - \frac{M_2(r^*+\sizeofboundary)^{2-n}}{R_2^{2-n} - (r^*+\sizeofboundary)^{2-n}}$ in the annular domain $B$, we have $u_2(x) = \frac{M_2}{R_2^{2-n} - (r^*+\sizeofboundary)^{2-n}}\norm{x}^{2-n} - \frac{M_2(r^*+\sizeofboundary)^{2-n}}{R_2^{2-n} - (r^*+\sizeofboundary)^{2-n}}$ identically in $B$. 

Similarly as in the case of $u_1$, we have 
\begin{equation}
\label{Normal_Derivative_for_u_2_n_geq_3}
\frac{\partial u_2}{\partial \nu}((r^*+\sizeofboundary)e_n) = -\frac{M_2}{R_2^{2-n} - (r^* + \sizeofboundary)^{2-n}}(2-n) (r^* + \sizeofboundary)^{1-n},
\end{equation}
where we have a minus sign as now the inward is the positive direction.

Now by \eqref{Normal_Derivative_for_u_1_n_geq_3}, \eqref{Normal_Derivative_for_u_2_n_geq_3}, and \eqref{The_First_Free_Boundary_Condition} in Theorem \ref{The_Free_Boundary_Condition}, we have 
\begin{equation}
\frac{\frac{M_1}{R_1^{2-n} - (r^*)^{2-n}}(2-n)(r^*)^{1-n}}{-\frac{M_2}{R_2^{2-n} - (r^* + \sizeofboundary)^{2-n}}(2-n) (r^* + \sizeofboundary)^{1-n}} = \frac{(r^*)^{1-n}}{(r^* + \sizeofboundary)^{1-n}}.
\end{equation}
Therefore, 
\begin{equation}
\frac{M_1}{R_1^{2-n} - (r^*)^{2-n}} = -\frac{M_2}{R_2^{2-n} - (r^* + \sizeofboundary)^{2-n}}.
\end{equation}
Rearrange the equation, and we obtain \eqref{The_Equation_Satisfied_by_r^*_in_nD}. 
Finally, we observed that $r^*$ is the solution of the equation $\phi_{R_1, R_2, \sizeofboundary, n}(r^*) = \frac{M_1}{M_2}$, and so in view of Lemma \ref{Phi_is_Strictly_Increasing}, $r^*$ is unique.  
\end{proof}

Having determined $r^*$ uniquely in terms of $M_1, M_2, R_2, R_2$, and $\sizeofboundary$, our next step is to study the asymptotic behavior of $r^*$ as $M_1 \rightarrow \infty$, $M_1 \rightarrow 0$ (for $M_2, R_1, R_2$, and $\sizeofboundary$ fixed), and $M_2 \rightarrow \infty$, $M_2 \rightarrow 0$ (for $M_1, R_1, R_2$, and $\sizeofboundary$ fixed). We have the following result:

\begin{lemma}
\label{Relation_Between_r^*_and_M_1_and_M_2}

We have the followings:
\begin{enumerate}
\item In any dimensions, for fixed $M_2, \sizeofboundary, R_1, R_2$, $r^* = r^*(M_1, M_2, \sizeofboundary, R_1, R_2)$ is increasing as a function of $M_1$, and
\begin{equation}
\begin{split}
\lim_{M_1 \rightarrow 0^+}r^*(M_1, M_2, \sizeofboundary, R_1, R_2) = R_1, \\
\lim_{M_1 \rightarrow \infty} r^*(M_1, M_2, \sizeofboundary, R_1, R_2) = R_2 -\sizeofboundary.    
\end{split}
\end{equation}
\item In any dimensions, for fixed $M_1, \sizeofboundary, R_1, R_2$, $r^* = r^*(M_1, M_2, \sizeofboundary, R_1, R_2)$ is decreasing as a function of $M_2$, and 
\begin{equation}
\begin{split}
\lim_{M_2 \rightarrow 0^+}r^*(M_1, M_2, \sizeofboundary, R_1, R_2) = R_2 - \sizeofboundary, \\
\lim_{M_2 \rightarrow \infty} r^*(M_1, M_2, \sizeofboundary, R_1, R_2) = R_1.  
\end{split}
\end{equation}
\end{enumerate}
\end{lemma} 
\begin{proof}

Note that $r^* = r^*(M_1, M_2, R_1, R_2, \sizeofboundary) = \phi^{-1}_{R_1, R_2, \sizeofboundary, n}(\frac{M_1}{M_2})$. Also, by Lemma \ref{Phi_is_Strictly_Increasing}, $\phi_{R_1, R_2, \sizeofboundary, n} : [R_1, R_2 -\sizeofboundary) \rightarrow [0,\infty)$ is strictly increasing, continuous, bijective, and $\lim_{x \rightarrow (R_2 - \sizeofboundary)^-} \phi_{R_1, R_2, \sizeofboundary, n}(x) = \infty$. Therefore, $\phi^{-1}_{R_1, R_2, \sizeofboundary, n}: [0, \infty) \rightarrow [R_1, R_2 - \sizeofboundary)$ is strictly increasing, continuous, bijective, $\lim_{x \rightarrow \infty}\phi_{R_1, R_2, \sizeofboundary, n}^{-1}(x) = R_2 - R_1$, and $\phi^{-1}_{R_1, R_2, \sizeofboundary, n}(0) = R_1$. Hence, $r^* = r^*(M_1, M_2, \sizeofboundary, R_1, R_2) = \phi^{-1}_{R_1, R_2, \sizeofboundary, n}(\frac{M_1}{M_2})$ is increasing as a function of $M_1$ (for fixed $M_2, \sizeofboundary, R_1, R_2$). Moreover, for fixed $M_2, \sizeofboundary, R_1, R_2$, we have
\begin{equation}
\begin{split}
\lim_{M_1 \rightarrow 0^+}r^*(M_1, M_2, \sizeofboundary, R_1, R_2) = R_1, \\
\lim_{M_1 \rightarrow \infty} r^*(M_1, M_2, \sizeofboundary, R_1, R_2) = R_2 -\sizeofboundary.    
\end{split}
\end{equation}
This proves (1). (2) can be proved similarly by using the fact that $r^* = r^*(M_1, M_2, \sizeofboundary, R_1, R_2) = \phi^{-1}_{R_1, R_2, \sizeofboundary, n}(\frac{M_1}{M_2})$.
\end{proof}

We give a short remark regarding Lemma \ref{Relation_Between_r^*_and_M_1_and_M_2}.

\begin{remark}

Lemma \ref{Relation_Between_r^*_and_M_1_and_M_2} matches with our intuition. Assume $M_1 \rightarrow \infty$, so from the point of view of the physical model, the first population in the inner layer grows stronger. This pushes the free boundary $\partial \{u_1 > 0\} \cap \Omega$ to the outer layer, so the radius $r^* = r^*(M_1, M_2, \sizeofboundary, R_1, R_2)$ increases. In the limiting case as $M_1 \rightarrow \infty$, $r^*$ converges to its maximal possible value, namely $R_2 - \sizeofboundary$. On the other hand, as $M_1 \rightarrow 0$, $r^*$ converges to its minimal possible value, namely $R_1$. The same idea applies to the case where the second population $M_2 \rightarrow 0$ and $M_2 \rightarrow \infty$.
\end{remark}

Now we proceed to show that as the interaction radius $\sizeofboundary \rightarrow 0^+$, the radius $r^* = r^*(M_1, M_2, \sizeofboundary, R_1, R_2)$ converges to $r_0 = r_0(M_1, M_2, R_1, R_2)$, where 
\begin{equation}
\begin{split}
    r_0 = r_0(M_1, M_2, R_1, R_2) = \begin{cases} R_1^{\frac{M_2}{M_1 + M_2}}R_2^{\frac{M_1}{M_1 + M_2}}, & n = 2, \\
    (\frac{M_1}{M_1 + M_2}\frac{1}{R_2^{n-2}} + \frac{M_2}{M_1 + M_2}\frac{1}{R_1^{n-2}})^{\frac{1}{2-n}}, & n \geq 3, 
    \end{cases}
\end{split}
\end{equation}
as mentioned in Remark \ref{Remark_about_Main_Result_1}. This result connects the theory of long-range interaction with the adjacent interaction.  

\begin{theorem}
\label{Relation_Between_r^*_and_Interaction_Distance}
 In any dimensions, for fixed $M_1, M_2, R_1, R_2$, $r^* = r^*(M_1, M_2, R, R_1, R_2)$ is decreasing as a function of $\sizeofboundary$, and
\begin{equation}
\begin{split}
\lim_{\sizeofboundary \rightarrow 0^+} r^*(M_1, M_2, \sizeofboundary, R_1, R_2) = r_0 = r_0(M_1, M_2, R_1, R_2),  
\end{split}
\end{equation}
where 
\begin{equation}
\begin{split}
    r_0 = r_0(M_1, M_2, R_1, R_2) = \begin{cases} R_1^{\frac{M_2}{M_1 + M_2}}R_2^{\frac{M_1}{M_1 + M_2}}, & n = 2, \\
    (\frac{M_1}{M_1 + M_2}\frac{1}{R_2^{n-2}} + \frac{M_2}{M_1 + M_2}\frac{1}{R_1^{n-2}})^{\frac{1}{2-n}}, & n \geq 3. 
    \end{cases}
\end{split}
\end{equation}
\end{theorem}
\begin{proof}
Recall that $r^* = r^*(M_1, M_2, \sizeofboundary, R_1, R_2) = \phi^{-1}_{R_1, R_2, \sizeofboundary,n}(\frac{M_1}{M_2})$, so by Lemma \ref{Phi_is_Strictly_Increasing} (2), it follows that $r^* = r^*(M_1, M_2, \sizeofboundary, R_1, R_2) = \phi^{-1}_{R_1, R_2, \sizeofboundary,n}(\frac{M_1}{M_2})$ is decreasing as a function of $\sizeofboundary$.

Also, by Lemma \ref{Phi_is_Strictly_Increasing} (3), we have $\lim_{\sizeofboundary \rightarrow 0^+} r^*(M_1, M_2, \sizeofboundary, R_1, R_2) = \lim_{\sizeofboundary \rightarrow 0^+}\phi^{-1}_{R_1, R_2, \sizeofboundary, n}(\frac{M_1}{M_2}) = \phi^{-1}_{R_1, R_2, 0, n}(\frac{M_1}{M_2})$. By a direct computation, we have $\phi_{R_1, R_2, 0, 2}(R_1^{\frac{M_2}{M_1 + M_2}}R_2^{\frac{M_1}{M_1 + M_2}}) = \frac{M_1}{M_2}$, and $\phi_{R_1, R_2, 0, n}((\frac{M_1}{M_1 + M_2}\frac{1}{R_2^{n-2}} + \frac{M_2}{M_1 + M_2}\frac{1}{R_1^{n-2}})^{\frac{1}{2 - n}}) = \frac{M_1}{M_2}$. Therefore, we have 
\begin{equation}
\begin{split}
\phi^{-1}_{R_1, R_2, 0, n}\left(\frac{M_1}{M_2}\right) = \begin{cases} R_1^{\frac{M_2}{M_1 + M_2}}R_2^{\frac{M_1}{M_1 + M_2}}, & n = 2, \\
(\frac{M_1}{M_1 + M_2}\frac{1}{R_2^{n-2}} + \frac{M_2}{M_1 + M_2}\frac{1}{R_1^{n-2}})^{\frac{1}{2 - n}}, & n \geq 3.
\end{cases}
\end{split}
\end{equation}
Therefore, by combining the above result, and we get
\begin{equation}
\begin{split}
\lim_{\sizeofboundary \rightarrow 0^+} r^*(M_1, M_2, \sizeofboundary, R_1, R_2) = \lim_{\sizeofboundary \rightarrow 0^+}\phi^{-1}_{R_1, R_2, \sizeofboundary, n}\left(\frac{M_1}{M_2}\right) 
\\ = \phi^{-1}_{R_1, R_2, 0, n}\left(\frac{M_1}{M_2}\right) = \begin{cases} R_1^{\frac{M_2}{M_1 + M_2}}R_2^{\frac{M_1}{M_1 + M_2}}, & n = 2, \\
(\frac{M_1}{M_1 + M_2}\frac{1}{R_2^{n-2}} + \frac{M_2}{M_1 + M_2}\frac{1}{R_1^{n-2}})^{\frac{1}{2 - n}}, & n \geq 3.
\end{cases}
\end{split}
\end{equation}
\end{proof}
{\bf Proof of Theorem \ref{Main_Result_1}:}
For (1), we note that by Lemma \ref{Existence_of_r^*}, there is a $r^* \in [R_1, R_2 -\sizeofboundary]$ such that 
\begin{equation}
\begin{split}
\partial \{u_1 > 0\} \cap \Omega = \{x \in \mathbb{R}^n: \norm{x} = r^*\}, \\
\partial \{u_2 > 0\} \cap \Omega = \{x \in \mathbb{R}^n: \norm{x} = r^* + \sizeofboundary\}.
\end{split}
\end{equation}
Moreover, all the free boundary points are regular. Such a $r^*$ is unique by Lemma \ref{Analysis_of_r^*_in_2D} and Lemma \ref{Analysis_of_r^*_in_nD}. This establishes (1). (2) follows from Lemma \ref{Analysis_of_r^*_in_2D}. (3) follows from Lemma \ref{Analysis_of_r^*_in_nD}. (4) and (5) follow from Lemma \ref{Relation_Between_r^*_and_M_1_and_M_2}, and (6) follows from Theorem \ref{Relation_Between_r^*_and_Interaction_Distance}.
$\Box$

We look at some examples of $r^* = r^*(M_1, M_2, \sizeofboundary, R_1, R_2)$.

\begin{example}
\label{Example_in_Which_M_1_Equals_M_2}
{\rm 
Consider the case in which $n = 2$ and $M_1 = M_2$, so by Theorem \ref{Main_Result_1} (2), one has 
\begin{equation}
    \frac{\log (\frac{r^*}{R_1})}{\log (\frac{R_2}{r^* + \sizeofboundary})} - 1 = 0,
\end{equation}
so $\frac{r^*}{R_1} = \frac{R_2}{r^* + \sizeofboundary}$. Note that this is a quadratic equation $(r^*)^2 + \sizeofboundary r^* - R_1R_2 = 0$ in $r^*$. Solve the quadratic equation for $r^*$ to obtain 
\begin{equation}
r^* = r^*(M_1, M_2, \sizeofboundary, R_1, R_2) = -\frac{1}{2}\sizeofboundary + \frac{1}{2}\sqrt{\sizeofboundary^2 + 4R_1R_2}.
\end{equation}
Observe that as $\sizeofboundary \rightarrow 0$, $r^* = r^*(M_1, M_2, \sizeofboundary, R_1, R_2) \rightarrow \sqrt{R_1R_2}$, which is consistent with Theorem \ref{Main_Result_1} (6). 
}
\end{example}

\section{One Parameter Families of Annular Domains and Boundary Data in Dimension Two}

This section is dedicated to the proof of Theorem \ref{Main_Result_2}. Recall that we study the one-parameter family of elliptic system \eqref{main_problems_rewritten_one_parameter_family}. We first consider the case in which $n = 2$.   

We assume that $(u^{\epsilon}_{1,t}, u^{\epsilon}_{2,t}) \in C^2(\Omega_t; \mathbb{R}^2) \cap C^0(\overline{\Omega_t}; \mathbb{R}^2)$ is a solution of \eqref{main_problems_rewritten_one_parameter_family} for any $\epsilon > 0, t \in I$, and suppose that $\{(u^{\epsilon}_{1,t}, u^{\epsilon}_{2,t})\}_{\epsilon > 0}$ converges along a subsequence as $\epsilon \rightarrow 0^+$ to $(u_{1,t}, u_{2,t})$. We also assume that the mapping $I \rightarrow \mathbb{R}$ given by $t \mapsto R_{1,t}$, $t \mapsto R_{2,t}$, $t \mapsto \sizeofboundary_{t}$, $t \mapsto {M}_{1,t}$, $t \mapsto {M}_{2,t}$ are all differentiable on $I$. Therefore, by Theorem \ref{Main_Result_1}, for any $t \in I$, there is a unique $r^* = r^*_{t} \in [R_{1,t}, R_{2,t} - \sizeofboundary_t]$ such that 
\begin{equation}
\begin{split}
\partial \{u_{1,t} > 0\} \cap \Omega_t = \{\norm{x} = r^*_{t}\},
\\ \partial \{u_{2,t} > 0\} \cap \Omega_t = \{\norm{x} = r^*_{t} + \sizeofboundary_t\}.
\end{split}
\end{equation}

We first compute the derivative of $r^*_t$.

\begin{lemma}
\label{Derivative_of_r^*_t}

Assume that $n = 2$, and let the setting be as above. Then the mapping $I \rightarrow \mathbb{R}$ given by $t \mapsto r^*_t$ is also differentiable on $I$, with 
\begin{equation}
\label{The_Main_Computational_Equation}
\left(\frac{d}{dt}r^*_t\right)\left(\frac{1}{x_{2,t}r^*_t} + \frac{x_{1,t}}{x_{2,t}^2(r^*_t + \sizeofboundary_t)}\right) = \frac{d}{dt}\left(\frac{M_{1,t}}{M_{2,t}}\right) + \frac{\frac{d}{dt}R_{1,t}}{R_{1,t}}\frac{1}{x_{2,t}} + \frac{x_{1,t}}{x_{2,t}^2}\frac{\frac{d}{dt}R_{2,t}}{R_{2,t}} - \frac{x_{1,t}\frac{d}{dt}\sizeofboundary_t}{x_{2,t}^2(r^*_t + \sizeofboundary_t)},
\end{equation}
where 
\begin{equation}
    x_{1,t} = \log(\frac{r^*_t}{R_{1,t}}) > 0, x_{2,t} = \log(\frac{R_{2,t}}{r^*_{t} + \sizeofboundary_t}) > 0
\end{equation}
for all $t \in I$.  
\end{lemma}
\begin{proof}
We recall that by Theorem \ref{Main_Result_1} (2), one has 
\begin{equation}
    \frac{x_{1,t}}{x_{2,t}} - \frac{M_{1,t}}{M_{2,t}} = 0
\end{equation}
for all $t \in I$. Take the derivative on both sides of the equation with respect to $t$, and we obtain that 
\begin{equation}
\label{Derivative_on_Both_Sides_Dimension_Two}
    \frac{(\frac{dx_{1,t}}{dt})x_{2,t} - x_{1,t}(\frac{x_{2,t}}{dt})}{x_{2,t}^2} - \frac{d}{dt}\left(\frac{M_{1,t}}{M_{2,t}}\right)  = 0.
\end{equation}
By a direct computation, 
\begin{equation}
\label{Derivative_of_x_1,t}
\begin{array}{rl}
\begin{split}
\frac{d}{dt}x_{1,t} & = \frac{1}{\left(\frac{r^*_t}{R_{1,t}}\right)}\frac{d}{dt}\left(\frac{r^*_t}{R_{1,t}}\right)
\\ &= \frac{R_{1,t}}{r^*_t}\frac{(\frac{d}{dt}r^*_t)(R_{1,t}) - (r^*_t)(\frac{d}{dt}R_{1,t})}{R_{1,t}^2}
\\ &= \frac{(\frac{d}{dt}r^*_t)(R_{1,t}) - (r^*_t)(\frac{d}{dt}R_{1,t})}{r^*_tR_{1,t}}
\end{split}
\end{array}
\end{equation} 
Similarly, by a direct computation, 
\begin{equation}
\label{Derivative_of_x_2,t}
\begin{split}
\frac{d}{dt}x_{2,t} = \frac{(\frac{d}{dt}R_{2,t})(r^*_t + \sizeofboundary_t) - (R_{2,t})(\frac{d}{dt}(r^*_t + \sizeofboundary_t))}{R_{2,t}(r^*_t + \sizeofboundary_t)}
\end{split}
\end{equation}
Now plug \eqref{Derivative_of_x_1,t} and \eqref{Derivative_of_x_2,t} into \eqref{Derivative_on_Both_Sides_Dimension_Two}, and we obtain 
\begin{equation}
    \frac{(\frac{(\frac{d}{dt}r^*_t)(R_{1,t}) - (r^*_t)(\frac{d}{dt}R_{1,t})}{r^*_tR_{1,t}})x_{2,t} - x_{1,t}(\frac{(\frac{d}{dt}R_{2,t})(r^*_t + \sizeofboundary_t) - (R_{2,t})(\frac{d}{dt}(r^*_t + \sizeofboundary_t))}{R_{2,t}(r^*_t + \sizeofboundary_t)})}{x_{2,t}^2} - \frac{d}{dt}\left(\frac{M_{1,t}}{M_{2,t}}\right)  = 0. \nonumber 
\end{equation}
Simplify the left hand side of the above equality, and we obtain 
\begin{equation}
    \frac{1}{x_{2,t}}(\frac{\frac{d}{dt}r^*_t}{r^*_t} - \frac{\frac{d}{dt}R_{1,t}}{R_{1,t}}) - \frac{x_{1,t}}{x_{2,t}^2}(\frac{\frac{d}{dt}R_{2,t}}{R_{2,t}} - \frac{\frac{d}{dt}r^*_t + \frac{d}{dt}\sizeofboundary_t}{r^*_t + \sizeofboundary_t}) - \frac{d}{dt}\left(\frac{M_{1,t}}{M_{2,t}}\right)  = 0. \nonumber 
\end{equation}
Rearrange the above equation, and we obtain
\begin{equation}
    \left(\frac{d}{dt}r^*_t\right)\left(\frac{1}{x_{2,t}r^*_t}\right) + \frac{x_{1,t}}{x_{2,t}^2}\frac{\frac{d}{dt}r^*_t + \frac{d}{dt}\sizeofboundary_{t}}{r^*_t + \sizeofboundary_{t}} = \frac{d}{dt}\left(\frac{M_{1,t}}{M_{2,t}}\right) + \frac{\frac{d}{dt}R_{1,t}}{R_{1,t}}\frac{1}{x_{2,t}} + \frac{x_{1,t}}{x_{2,t}^2}\frac{\frac{d}{dt}R_{2,t}}{R_{2,t}}. \nonumber
\end{equation}
                                                                                                  
By a simple calculation, we have
\begin{equation}
    \left(\frac{d}{dt}r^*_t\right)\left(\frac{1}{x_{2,t}r^*_t} + \frac{x_{1,t}}{x_{2,t}^2(r^*_t + \sizeofboundary_t)}\right) = \frac{d}{dt}\left(\frac{M_{1,t}}{M_{2,t}}\right) + \frac{\frac{d}{dt}R_{1,t}}{R_{1,t}}\frac{1}{x_{2,t}} + \frac{x_{1,t}}{x_{2,t}^2}\frac{\frac{d}{dt}R_{2,t}}{R_{2,t}} - \frac{x_{1,t}\frac{d}{dt}\sizeofboundary_t}{x_{2,t}^2(r^*_t + \sizeofboundary_t)}. \nonumber
\end{equation}
This yields \eqref{The_Main_Computational_Equation}, as desired.
\end{proof}

As an immediate application of Lemma \ref{Derivative_of_r^*_t}, we have the following result.

\begin{corollary}
\label{r^*_t_is_constant_if_ratio_is_constant}

Assume that $n = 2$, and let the setting be as above. Suppose that $R_{1,t}, R_{2,t}$, and $\sizeofboundary_t$ are constants in $t \in I$. Then $r^*_t$ is constant in $t \in I$ if $\frac{M_{1,t}}{M_{2,t}}$ is constant in $t \in I$.
\end{corollary}
\begin{proof}
One can see this immediately from Theorem \ref{Main_Result_1} (2). Here we deduce the result from Lemma \ref{Derivative_of_r^*_t}. By Lemma \ref{Derivative_of_r^*_t}, we have
\begin{equation}
    \left(\frac{dr^*_t}{dt}\right)\left(\frac{1}{x_{2,t}r^*_t} + \frac{x_{1,t}}{x_{2,t}^2r^*_t}\right) = \frac{d}{dt}\left(\frac{M_{1,t}}{M_{2,t}}\right) = 0,
\end{equation}
so $r^*_t$ is constant in $t \in I$. 
\end{proof}

Corollary \ref{r^*_t_is_constant_if_ratio_is_constant} says that for fixed $R_{1,t}, R_{2,t}$, and $\sizeofboundary_t$, the radius of the free boundary $r^*_t$ is only determined by the ratio $\frac{M_{1,t}}{M_{2,t}}$. In particular, if $M_{1,t}$ and $M_{2,t}$ changes in $t \in I$ in such a way that $\frac{M_{1,t}}{M_{2,t}}$ is constant in $t \in I$, $r^*_t$ is constant, as shown in the next example.

\begin{example}

We take $I = (0,\infty)$, and assume that $R_{1,t} = R_1, R_{2,t} = R_2$, and $\sizeofboundary_t = \sizeofboundary$ for all $t \in (0, \infty)$. Also, we take $M_{1,t} = M_{2,t} = t^2 + 1$ for all $t \in (0, \infty)$. By Example \ref{Example_in_Which_M_1_Equals_M_2} and Corollary \ref{r^*_t_is_constant_if_ratio_is_constant}, it follows that 
\begin{equation}
r^*_t = -\frac{1}{2}\sizeofboundary + \frac{1}{2}\sqrt{\sizeofboundary^2 + 4R_1R_2}.
\end{equation}
for all $t \in (0, \infty)$. Therefore, although $M_{1,t}$ and $M_{2,t}$ increases in $t \in (0,\infty)$, the radius $r^*_t$ remains invariant. 
\end{example}

Our next goal is to study the conditions on $R_{1,t}$ and $R_{2,t}$ for which $r^*_t$ remains invariant. We have the following result.

\begin{proposition}
\label{Equilibrium_for_R_1_and_R_2}

Assume that $n = 2$, and let the setting be as above. Assume $\frac{M_{1,t}}{M_{2,t}}$ is a constant $\alpha$ for all $t \in I$, and that $\sizeofboundary_t$ is also constant in $t \in I$. Then $r^*_t$ is constant in $t \in I$ if 
\begin{equation}
\label{Key_Equation_for_R_1_and_R_2}
    \frac{\frac{d}{dt}(\log R_{1,t})}{\frac{d}{dt}(\log R_{2,t})} = -\alpha
\end{equation}
for all $t \in I$.
\end{proposition}   
\begin{proof}
Observe that by Theorem \ref{Main_Result_1} (2), $\alpha = \frac{M_{1,t}}{M_{2,t}} = \frac{x_{1,t}}{x_{2,t}}$, where 
\begin{equation}
    x_{1,t} = \log(\frac{r^*_t}{R_{1,t}}) > 0, x_{2,t} = \log(\frac{R_{2,t}}{r^*_{t} + \sizeofboundary_t}) > 0
\end{equation}
for all $t \in I$. Therefore, we have from \eqref{Key_Equation_for_R_1_and_R_2} that 
\begin{equation}
    \frac{\frac{d}{dt}(\log R_{1,t})}{\frac{d}{dt}(\log R_{2,t})} = -\frac{x_{1,t}}{x_{2,t}}.
\end{equation}
Therefore, 
\begin{equation}
    \frac{\frac{1}{R_{1,t}}\frac{d}{dt}(R_{1,t})}{\frac{1}{R_{2,t}}\frac{d}{dt}( R_{2,t})} = -\frac{x_{1,t}}{x_{2,t}},
\end{equation}
so 
\begin{equation}
\label{Key_Equation_Relating_Derivative_of_R_1_and_R_2_Dimension_Two}
\frac{x_{2,t}}{R_{1,t}}\left(\frac{d}{dt}R_{1,t}\right) + \frac{x_{1,t}}{R_{2,t}}\left(\frac{d}{dt}R_{2,t}\right) = 0.
\end{equation}
Divide the equation \eqref{Key_Equation_Relating_Derivative_of_R_1_and_R_2_Dimension_Two} by $x_{2,t}^2$, and we obtain 
\begin{equation}
\label{Key_Equation_Relating_Derivative_of_R_1t_and_R_2t_Dimension_two}
\frac{1}{R_{1,t}x_{2,t}}\left(\frac{dR_{1,t}}{dt}\right) + \frac{x_{1,t}}{R_{2,t}x_{2,t}^2}\left(\frac{dR_{2,t}}{dt}\right) = 0.
\end{equation}
Now recall that $\frac{M_{1,t}}{M_{2,t}}$ is a constant $\alpha$ for all $t \in I$, and that $\sizeofboundary_t$ is a constant in $t \in I$, so by Lemma \ref{Derivative_of_r^*_t} we have
\begin{equation}
\label{Application_of_Lemma_4_1_to_Proposition_4_4_Dimension_Two}
     \left(\frac{d}{dt}r^*_t\right)\left(\frac{1}{x_{2,t}r^*_t} + \frac{x_{1,t}}{x_{2,t}^2(r^*_t + \sizeofboundary_t)}\right) = \frac{\frac{d}{dt}R_{1,t}}{R_{1,t}}\frac{1}{x_{2,t}} + \frac{x_{1,t}}{x_{2,t}^2}\frac{\frac{d}{dt}R_{2,t}}{R_{2,t}}.
\end{equation}
Plug \eqref{Key_Equation_Relating_Derivative_of_R_1t_and_R_2t_Dimension_two} into \eqref{Application_of_Lemma_4_1_to_Proposition_4_4_Dimension_Two}, and we obtain 
\begin{equation}
    \left(\frac{dr^*_t}{dt}\right) = 0,
\end{equation}
so $r^*_t$ is a constant in $t \in I$.
\end{proof}

We consider a related example.

\begin{example}

Consider the case where $n = 2$, $I = (0,\frac{1}{2}\ln 2)$, $M_{1,t} = 2$, $M_{2,t} = 1$, $R_{1,t} = (2\sizeofboundary)e^{-2t}$, $R_{2,t} = (5\sizeofboundary)e^t$, and $\sizeofboundary_t = \sizeofboundary$ for all $t \in I = (0,\frac{1}{2}\ln 2)$. Clearly,
\begin{equation}
    \frac{\frac{d}{dt}(\log R_{1,t})}{\frac{d}{dt}(\log (R_{2,t})} = \frac{-2}{1} = -\frac{M_{1,t}}{M_{2,t}}
\end{equation}
for all $t \in I$. By Proposition \ref{Equilibrium_for_R_1_and_R_2}, it follows that $r^*_t$ is constant for all $t \in I = (0,\frac{1}{2}\ln 2)$.

 
\end{example}

We are also interested in studying $\frac{d}{dt}r^*_t$ in terms of $\frac{d}{dt}\sizeofboundary_t$, for fixed $M_{1,t}, M_{2,t}, R_{1,t}, R_{2,t}$. We have the following result.
\begin{corollary}
\label{Relation_between_logr^*_t_and_R_t}
Assume that $n = 2$, and let the setting be as above. Assume that $R_{1,t}, R_{2,t}$ are constants in $t \in I$, and suppose that $\frac{M_{1,t}}{M_{2,t}}$ is a constant $\alpha$ for all $t \in I$. Then we have 
\begin{equation}
\label{Key_Equation_Relating_Derivative_r*t_R_t}
    \frac{dr^*_t}{dt} = -\frac{\alpha}{(1+\alpha) + \frac{\sizeofboundary_t}{r^*_t}}\frac{d\sizeofboundary_t}{dt}
\end{equation}
\end{corollary}
\begin{proof}
By Lemma \ref{Derivative_of_r^*_t}, we have 
\begin{equation}
\left(\frac{d}{dt}r^*_t\right)\left(\frac{1}{x_{2,t}r^*_t} + \frac{x_{1,t}}{x_{2,t}^2(r^*_t + \sizeofboundary_t)}\right) = - \frac{x_{1,t}\frac{d}{dt}\sizeofboundary_t}{x_{2,t}^2(r^*_t + \sizeofboundary_t)},
\end{equation}
where $x_{1,t} = \log(\frac{r^*_t}{R_{1,t}}) > 0, x_{2,t} = \log(\frac{R_{2,t}}{r^*_{t} + \sizeofboundary_t}) > 0$.
Thus, 
\begin{equation}
\label{Equation_for_Computing_the_Ratio_of_Quotient}
    \frac{\frac{d}{dt}r^*_t}{\frac{d}{dt}\sizeofboundary_t} = \frac{-\frac{x_{1,t}}{x_{2,t}^2}\frac{1}{r^*_t + \sizeofboundary_t}}{\frac{1}{x_{2,t}r^*_t} + \frac{x_{1,t}}{x_{2,t}^2(r^*_t + \sizeofboundary_t)}}.
\end{equation}
Simplify the right hand side of \eqref{Equation_for_Computing_the_Ratio_of_Quotient} and use the identity $\frac{x_{1,t}}{x_{2,t}} = \frac{M_{1,t}}{M_{2,t}} = \alpha$, and we obtain \eqref{Key_Equation_Relating_Derivative_r*t_R_t}, as desired.
\end{proof}

By combining all the results above, we can now prove the second main result.

{\bf Proof of Theorem \ref{Main_Result_2}:}

This follows from Lemma \ref{Derivative_of_r^*_t}, Corollary \ref{r^*_t_is_constant_if_ratio_is_constant}, and Proposition \ref{Equilibrium_for_R_1_and_R_2}.
$\Box$

\section{One Parameter Families of Annular Domains and Boundary Data in Higher Dimensions}

This section is dedicated to the proof of Theorem \ref{Main_Result_3}. Recall that we study the one-parameter family of elliptic system \eqref{main_problems_rewritten_one_parameter_family}. We now consider the case in which $n \geq 3$.   

We assume that $(u^{\epsilon}_{1,t}, u^{\epsilon}_{2,t}) \in C^2(\Omega_t; \mathbb{R}^2) \cap C^0(\overline{\Omega_t}; \mathbb{R}^2)$ is a solution of \eqref{main_problems_rewritten_one_parameter_family} for any $\epsilon > 0, t \in I$, and suppose that $\{(u^{\epsilon}_{1,t}, u^{\epsilon}_{2,t})\}_{\epsilon > 0}$ converges along a subsequence as $\epsilon \rightarrow 0^+$ to $(u_{1,t}, u_{2,t})$. We also assume that the mapping $I \rightarrow \mathbb{R}$ given by $t \mapsto R_{1,t}$, $t \mapsto R_{2,t}$, $t \mapsto \sizeofboundary_{t}$, $t \mapsto {M}_{1,t}$, $t \mapsto {M}_{2,t}$ are all differentiable on $I$. Therefore, by Theorem \ref{Main_Result_1}, for any $t \in I$, there is a unique $r^* = r^*_{t} \in [R_{1,t}, R_{2,t} - \sizeofboundary_t]$ such that 
\begin{equation}
\begin{split}
\partial \{u_{1,t} > 0\} \cap \Omega_t = \{\norm{x} = r^*_{t}\},
\\ \partial \{u_{2,t} > 0\} \cap \Omega_t = \{\norm{x} = r^*_{t} + \sizeofboundary_t\}.
\end{split}
\end{equation}

We first compute the derivative of $r^*_t$. 

\begin{lemma}
\label{Derivative_of_r^*_t_Higher_Dimensions}

Assume that $n \geq 3$, and let the setting be as above. Then the mapping $I \rightarrow \mathbb{R}$ given by $t \mapsto r^*_t$ is also differentiable on $I$, with 
\begin{equation}
\label{The_Main_Computational_Equation_Higher_Dimensions}
\begin{split}
\left(\frac{dr^*_t}{dt}\right)(n-2)[\frac{1}{x_{2,t}}(r^*_t)^{1-n} + \frac{x_{1,t}}{x_{2,t}^2}(r^*_t + \sizeofboundary_t)^{1-n}] = \frac{d}{dt}\left(\frac{M_{1,t}}{M_{2,t}}\right) 
\\+ (n-2)\frac{1}{x_{2,t}}R^{1-n}_{1,t}\left(\frac{dR_{1,t}}{dt}\right) + (n-2)\frac{x_{1,t}}{x_{2,t}^2}R^{1-n}_{2,t}\left(\frac{dR_{2,t}}{dt}\right) - (n-2)\frac{x_{1,t}}{x_{2,t}^2}(r^*_t + \sizeofboundary_t)^{1-n}\left(\frac{d\sizeofboundary_t}{dt}\right),
\end{split}
\end{equation}
where 
\begin{equation}
    x_{1,t} = R_{1,t}^{2-n} - (r^*_t)^{2-n} > 0, x_{2,t} = (r^*_t + \sizeofboundary_t)^{2-n} - R_{2,t}^{2-n} > 0
\end{equation}
for all $t \in I$.  
\end{lemma}

\begin{proof}
We recall that by Theorem \ref{Main_Result_1} (2), one has 
\begin{equation}
    \frac{x_{1,t}}{x_{2,t}} - \frac{M_{1,t}}{M_{2,t}} = 0
\end{equation}
for all $t \in I$. Take the derivative on both sides of the equation with respect to $t$, and we obtain that 
\begin{equation}
\label{Derivative_on_Both_Sides}
    \frac{(\frac{dx_{1,t}}{dt})x_{2,t} - x_{1,t}(\frac{x_{2,t}}{dt})}{x_{2,t}^2} - \frac{d}{dt}\left(\frac{M_{1,t}}{M_{2,t}}\right)  = 0.
\end{equation}
By a direct computation, 

\begin{equation}
\label{Derivative_of_x_1,t_Higher_Dimensions}
\begin{split}
\frac{d}{dt}x_{1,t} = (2-n)R_{1,t}^{1-n}(\frac{d}{dt}R_{1,t}) - (2-n)(r^*_t)^{1-n}(\frac{d}{dt}r^*_t),
\end{split}
\end{equation}
\begin{equation}
\label{Derivative_of_x_2,t_Higher_Dimensions}
\frac{d}{dt}x_{2,t} = (2-n)(r^*_t + \sizeofboundary_t)^{1-n}(\frac{d}{dt}r^*_t + \frac{d}{dt}\sizeofboundary_t) - (2-n)R_{2,t}^{1-n}(\frac{d}{dt}R_{2,t}).    
\end{equation}

Plug \eqref{Derivative_of_x_1,t_Higher_Dimensions} and \eqref{Derivative_of_x_2,t_Higher_Dimensions} into \eqref{Derivative_on_Both_Sides}, and solve for $\frac{dr^*_t}{dt}$. We obtain \eqref{The_Main_Computational_Equation_Higher_Dimensions}, as desired. 

\end{proof}

As an immediate application of Lemma \ref{Derivative_of_r^*_t_Higher_Dimensions}, we have the following result.

\begin{corollary}
\label{r^*_t_is_constant_if_ratio_is_constant_Higher_Dimensions}

Assume that $n \geq 3$, and let the setting be as above. Suppose that $R_{1,t}, R_{2,t}$, and $\sizeofboundary_t$ are constants in $t \in I$. Then $r^*_t$ is constant in $t \in I$ if $\frac{M_{1,t}}{M_{2,t}}$ is constant in $t \in I$.
\end{corollary}
\begin{proof}
One can see this immediately from Theorem \ref{Main_Result_1} (2). Here we deduce the result from Lemma \ref{Derivative_of_r^*_t_Higher_Dimensions}. By Lemma \ref{Derivative_of_r^*_t_Higher_Dimensions}, we have
\begin{equation}
\frac{d}{dt}(r^*_t) = 0.
\end{equation}
so $r^*_t$ is constant in $t \in I$. 
\end{proof}

Corollary \ref{r^*_t_is_constant_if_ratio_is_constant_Higher_Dimensions} says that for fixed $R_{1,t}, R_{2,t}$, and $\sizeofboundary_t$, the radius of the free boundary $r^*_t$ is only determined by the ratio $\frac{M_{1,t}}{M_{2,t}}$. In particular, if $M_{1,t}$ and $M_{2,t}$ changes in $t \in I$ in such a way that $\frac{M_{1,t}}{M_{2,t}}$ is constant in $t \in I$, $r^*_t$ is constant.

Our next goal is to study the conditions on $R_{1,t}$ and $R_{2,t}$ for which $r^*_t$ remains invariant. We have the following result.

\begin{proposition}
\label{Equilibrium_for_R_1_and_R_2_Higher_Dimensions}

Assume that $n \geq 3$, and let the setting be as above. Assume $\frac{M_{1,t}}{M_{2,t}}$ is a constant $\alpha$ for all $t \in I$, and that $\sizeofboundary_t$ is also constant in $t \in I$. Then $r^*_t$ is constant in $t \in I$ if 
\begin{equation}
\label{Key_Equation_for_R_1_and_R_2_Higher_Dimensions}
\frac{\frac{d}{dt}(R_{1,t}^{2-n})}{\frac{d}{dt}(R_{2,t}^{2-n})} = -\alpha
\end{equation}
for all $t \in I$.
\end{proposition}

\begin{proof}

Observe that by Theorem \ref{Main_Result_1} (2), $\alpha = \frac{M_{1,t}}{M_{2,t}} = \frac{x_{1,t}}{x_{2,t}}$, where 
\begin{equation}
    x_{1,t} = R_{1,t}^{2-n} - (r^*_t)^{2-n} > 0, x_{2,t} = (r^*_t + \sizeofboundary_t)^{2-n} - R_{2,t}^{2-n} > 0
\end{equation}
for all $t \in I$. Therefore, we have from \eqref{Key_Equation_for_R_1_and_R_2_Higher_Dimensions} that 
\begin{equation}
    \frac{\frac{d}{dt}(R_{1,t}^{2-n})}{\frac{d}{dt}(R_{2,t}^{2-n})} = -\frac{x_{1,t}}{x_{2,t}}.
\end{equation}
Therefore, 
\begin{equation}
    \frac{(2-n)R_{1,t}^{1-n}\frac{d}{dt}(R_{1,t})}{(2-n)R_{2,t}^{1-n}\frac{d}{dt}(R_{2,t})} = -\frac{x_{1,t}}{x_{2,t}},
\end{equation}
so 
\begin{equation}
\label{Key_Equation_Relating_Derivative_of_R_1_and_R_2}
{x_{2,t}}R_{1,t}^{1-n}\left(\frac{d}{dt}R_{1,t}\right) + {x_{1,t}}R_{2,t}^{1-n}\left(\frac{d}{dt}R_{2,t}\right) = 0.
\end{equation}
Divide the equation \eqref{Key_Equation_Relating_Derivative_of_R_1_and_R_2} by $x_{2,t}^2$, and we obtain 
\begin{equation}
\label{Key_Equation_Relating_Derivative_of_R_1t_and_R_2t}
\frac{R_{1,t}^{1-n}}{x_{2,t}}\left(\frac{dR_{1,t}}{dt}\right) + \frac{R_{2,t}^{1-n}x_{1,t}}{x_{2,t}^2}\left(\frac{dR_{2,t}}{dt}\right) = 0.
\end{equation}
Now recall that $\frac{M_{1,t}}{M_{2,t}}$ is a constant $\alpha$ for all $t \in I$, and that $\sizeofboundary_t$ is a constant in $t \in I$, so by Lemma \ref{Derivative_of_r^*_t_Higher_Dimensions} we have
\begin{equation}
\label{Application_of_Lemma_4_1_to_Proposition_4_4}
\left(\frac{dr^*_t}{dt}\right)[\frac{1}{x_{2,t}}(r^*_t)^{1-n} + \frac{x_{1,t}}{x_{2,t}^2}(r^*_t + \sizeofboundary_t)^{1-n}] =  
\\ \frac{1}{x_{2,t}}R^{1-n}_{1,t}\left(\frac{dR_{1,t}}{dt}\right) + \frac{x_{1,t}}{x_{2,t}^2}R^{1-n}_{2,t}\left(\frac{dR_{2,t}}{dt}\right).
\end{equation}
Plug \eqref{Key_Equation_Relating_Derivative_of_R_1t_and_R_2t} into \eqref{Application_of_Lemma_4_1_to_Proposition_4_4}, and we obtain 
\begin{equation}
    \left(\frac{dr^*_t}{dt}\right) = 0,
\end{equation}
so $r^*_t$ is a constant in $t \in I$.  
\end{proof}

We are also interested in studying $\frac{d}{dt}r^*_t$ in terms of $\frac{d}{dt}\sizeofboundary_t$, for fixed $M_{1,t}, M_{2,t}, R_{1,t}, R_{2,t}$. We have the following result.

\begin{corollary}
\label{Relation_between_logr^*_t_and_R_t_Higher_Dimensions}
Assume that $n \geq 3$, and let the setting be as above. Assume that $R_{1,t}, R_{2,t}$ are constants in $t \in I$, and suppose that $\frac{M_{1,t}}{M_{2,t}}$ is a constant $\alpha$ for all $t \in I$. Then we have 
\begin{equation}
\label{Relation_between_logr^*_t_and_R_t_Higher_Dimensions_Equation}
\frac{d}{dt}(r^*_t) = \frac{-\alpha}{\alpha + (1+\frac{\sizeofboundary_t}{r^*_t})^{n-1}}\frac{d}{dt}(\sizeofboundary_t)
\end{equation}
\end{corollary}
\begin{proof}
By Lemma \ref{Derivative_of_r^*_t_Higher_Dimensions} and the fact that $\frac{d}{dt}R_{
1,t} = \frac{d}{dt}R_{2,t} = \frac{d}{dt}(\frac{M_{1,t}}{M_{2,t}}) = 0$, we have
\begin{equation}
\begin{split}
\left(\frac{dr^*_t}{dt}\right)(n-2)[\frac{1}{x_{2,t}}(r^*_t)^{1-n} + \frac{x_{1,t}}{x_{2,t}^2}(r^*_t + \sizeofboundary_t)^{1-n}] = - (n-2)\frac{x_{1,t}}{x_{2,t}^2}(r^*_t + \sizeofboundary_t)^{1-n}\left(\frac{d\sizeofboundary_t}{dt}\right),
\end{split}
\end{equation}
where 
\begin{equation}
    x_{1,t} = R_{1,t}^{2-n} - (r^*_t)^{2-n} > 0, x_{2,t} = (r^*_t + \sizeofboundary_t)^{2-n} - R_{2,t}^{2-n} > 0.
\end{equation}
Divide both sides of the above equation by $(n-2)[\frac{1}{x_{2,t}}(r^*_t)^{1-n} + \frac{x_{1,t}}{x_{2,t}^2}(r^*_t + \sizeofboundary_t)^{1-n}](\frac{d}{dt}\sizeofboundary_t)$, use the fact that $\frac{x_{1,t}}{x_{2,t}} = \frac{M_{1,t}}{M_{2,t}} = \alpha$, and simplify the resulting equation, and we obtain \eqref{Relation_between_logr^*_t_and_R_t_Higher_Dimensions_Equation}, as desired.
\end{proof}

By combining all the results above, we can now prove the second main result.

{\bf Proof of Theorem \ref{Main_Result_3}:}

This follows from Lemma \ref{Derivative_of_r^*_t_Higher_Dimensions}, Corollary \ref{r^*_t_is_constant_if_ratio_is_constant_Higher_Dimensions}, Proposition \ref{Equilibrium_for_R_1_and_R_2_Higher_Dimensions}.
$\Box$

\end{document}